\documentclass[11pt,a4paper,leqno]{article}
\usepackage{a4wide}
\usepackage{latexsym}
\usepackage{amsmath}
\usepackage{amssymb}
\usepackage{stackrel}   
\usepackage{color}  
\usepackage{graphicx}
\usepackage[linesnumbered,ruled,vlined]{algorithm2e}

\usepackage{authblk}

\usepackage{hyperref}

\newtheorem{defin}{Definition}
\newtheorem{lemma}{Lemma} 
\newtheorem{prop}{Proposition}
\newtheorem{theo}{Theorem}

\newenvironment{proof}{\medskip\par\noindent{\bf Proof}}{\hfill $\Box$
\medskip\par}

\newcommand{\C}{\mathbb{C}}
\newcommand{\N}{\mathbb{N}}
\newcommand{\R}{\mathbb{R}}

\begin{document}
\title{Formal solution and strong asymptotic expansion of some initial value problem in two complex time variables, infinite order irregular singularities and Mahler transforms}

\author[1]{Alberto Lastra}
\author[2]{St\'ephane Malek}
\affil[1]{Universidad de Alcal\'a, Dpto. F\'isica y Matem\'aticas, Alcal\'a de Henares, Madrid, Spain. {\tt alberto.lastra@uah.es}}
\affil[2]{University of Lille, Laboratoire Paul Painlev\'e, Villeneuve d'Ascq cedex, France. {\tt stephane.malek@univ-lille.fr}}


\date{}

\maketitle
\thispagestyle{empty}
{ \small \begin{center}
{\bf Abstract}
\end{center}

A formal solution to a partial differential equation of infinite order involving the action of Mahler-type operators in several variables is obtained. Such formal solutions is constructed by means of a fixed point argument applied to the integral representation of the deceleration operator in several variables that has been developed. 

The Laplace transform in two variables of the actual solution of a related problem in both Borel plane and Fourier space turns out to admit strong asymptotic expansion along adequate multidirections.

\smallskip

}
\bigskip \bigskip

\section{Introduction}

The main objective in the present work is to study formal solutions associated to a family of nonlinear initial value problems combining partial derivatives and Mahler operators in several complex time variables, together with the action of formal differential operators of infinite order. Given positive integers $k_j,\Delta_j,\alpha_j$, for $j=1,2$, the shape of the equations under consideration is of the form
\begin{multline}\label{epralintro}
\left[Q_1(\partial_z)-\cosh\left(\alpha_1(t_1^{k_1+1}\partial_{t_1})^{\Delta_1}\right)R_{1}(\partial_z)\right]\\
\times \left[Q_2(\partial_z)-\cosh\left(\alpha_2(t_2^{k_2+1}\partial_{t_2})^{\Delta_2}\right)R_{2}(\partial_z)\right]u(t_1,t_2,z)\\
=P(t_1,t_2,z,t_1^{k_1+1}\partial_{t_1},t_2^{k_2+1}\partial_{t_2},\partial_z,\{\mathfrak{m}_{t_1,\ell_4,\ell_5}\}_{(\ell_4,\ell_5)\in\mathcal{A}_1},\{\mathfrak{m}_{t_2,\ell_6,\ell_7}\}_{(\ell_6,\ell_7)\in\mathcal{A}_2})u(t_1,t_2,z)\\
+\tilde{Q}_1(\partial_z)u(t_1,t_2,z)\times \tilde{Q}_2(\partial_z)u(t_1,t_2,z)+f(t_1,t_2,z),
\end{multline}
under null initial data $u(t_1,0,z)\equiv u(0,t_2,z)\equiv 0$. Here, $\cosh(\cdot)$ defines an infinite order differential operator (see Section~\ref{sec3}), and $Q_j(X),R_j(X),\tilde{Q}_j(X)\in\C[X]$ for $j=1,2$ are polynomials. The element $P(t_1,t_2,z,V_1,V_2,V_3,\{W_{(\ell_4,\ell_5)}\}_{(\ell_4,\ell_5)\in\mathcal{A}_1},\{W_{(\ell_6,\ell_7)}\}_{(\ell_6,\ell_7)\in\mathcal{A}_2})$ is a polynomial in $t_1,t_2,V_1,V_2,V_3$, a linear map with respect to $W_{(\ell_4,\ell_5)}$ and $W_{(\ell_6,\ell_7)}$ for $(\ell_4,\ell_5)\in\mathcal{A}_1$ and $(\ell_6,\ell_7)\in\mathcal{A}_2$, for some finite sets $\mathcal{A}_j$ of pairs of nonnegative integers. In addition to this, $P(\cdot)$ is a bounded holomorphic function with respect to $z$ on some horizontal strip $H_{\beta}=\{z\in\C:|\hbox{Im}(z)|<\beta\}$, for some fixed $\beta>0$. The forcing term $f(t_1,t_2,z)$ is a polynomial in $t_1,t_2$ with bounded holomorphic coefficients defined on $H_{\beta}$.

The symbols $\mathfrak{m}_{t_1,\ell,\tilde{\ell}}$ and $\mathfrak{m}_{t_2,\ell,\tilde{\ell}}$ stand for Mahler-type operator in two variables acting as follows:
$$\mathfrak{m}_{t_1,\ell,\tilde{\ell}}u(t_1,t_2,z)=u(t_1^{\ell}t_2^{\tilde{\ell}},t_2,z),\qquad \mathfrak{m}_{t_2,\ell,\tilde{\ell}}u(t_1,t_2,z)=u(t_1,t_1^{\ell}t_2^{\tilde{\ell}},z).$$

\vspace{0.3cm}

The previous operators generalize the so-called Mahler operators in a two variable framework, associated to Mahler's transformations in several variables as discused in the classical textbook by K. Nishioka,~\cite{nis}. For more recent advances in this topic with major applications in transcendence theory, we refer to the paper~\cite{ada3} by B. Adamczewski and C. Faverjon.


The corresponding Mahler operators in one variable, in the form $\mathfrak{m}_{t,\ell}u(t)=u(t^{\ell})$ turn out to be an active field of research nowadays through Mahler equations
$$\sum_{k=0}^{n}a_k(z)y(z^{\ell^k})=0,$$
with $\ell\ge2,n\ge1$ being integer numbers, and $a_k\in\C(z)$ as it can be seen in~\cite{ada}. Their applications are of great interest, too. Further research on Galoisian aspects and hypertranscendence results can also be found in~\cite{ada2}, and from an algebraic point of view in terms of Hahn series in~\cite{fav,roq,roq2}.

The combined appearance of Mahler operators and other operators in a functional equation is also of interest:~\cite{sasi} in the form of coupled systems of linear differential equations together with Mahler transforms;~\cite{ou2} in combination with difference operators, also under more generality in~\cite{ya2}. We also refer to the works~\cite{ma24} and~\cite{lama4} in the framework of partial differential equations and $q$-difference equations involving an infinite order irregular singularity, respectively.

In the last two mentioned works, the action of an infinite order irregular singularity at the origin in one complex time variable is the origin of some outstanding phenomena: the formal solution to the problem gives rise to an analytic function which does not provide an analytic solution to the initial problem. In the present study, an analogous phenomenon will arise when considering two complex time variables (see the remark after Theorem~\ref{teopral}).

Compared to~\cite{ma24}, the use of infinite order operators remains still mandatory in the construction of an appropriate formal solution to (\ref{epralintro}), however the orders $\Delta_1,\Delta_2\ge1$ of the irregular differential operators $(t_1^{k_1+1}\partial_{t_1})^{\Delta_1},(t_2^{k_2+1}\partial_{t_2})^{\Delta_2}$ involved need to be taken larger than 2 in general (see the requirements (\ref{e115}) for more details). In~\cite{ma24}, irregular operators of order 2, $(t^{k+1}\partial_t)^2$, were sufficient to build up formal solutions to which a Borel-Laplace procedure could be applied.

\vspace{0.3cm}

The main advance in this work is twofold. 

On the one hand, the appearance of Mahler-type operators in two complex time variables in the functional equation under study is understood via Majima's asymptotic expansions approach. Such asymptotic expansions, developped by H. Majima~\cite{majima1,majima2}, are able to explain the asymptotic behavior of certain holomorphic function in two variables (with coefficients in some complex Banach space) in terms of the formal solution to (\ref{epralintro}) when approaching $0\in\C^2$. A brief introduction of the concept of Majima's strong asymptotic expansions has been included in Appendix II (Section~\ref{sec877}). We also provide sufficient conditions for a Laplace transform in two variables to possess such asymptotic expansions. 

We refer to the main result in this work, Theorem~\ref{teopral}, which states:
\begin{itemize}
\item[(i)] the explicit form of the formal solution to (\ref{epralintro}), and
\item[(ii)] the existence of a holomorphic function in the pair of complex variables $(t_1,t_2)$ with values in a Banach space of continuous functions with exponential decay on the real line, defined in the product of two bounded sectors with adequate bisecting multidirection $(d_1,d_2)\in\R^2$, say $\tilde{S}_{d_1}\times \tilde{S}_{d_2}$. Such function, determined by means of a two-level version of Laplace transform, admits on $\tilde{S}_{d_1}\times \tilde{S}_{d_2}$ a strong asymptotic expansion of Gevrey multi-order $(k_1,k_2)$ whose associated set of data comprise coefficients of the formal solution to (\ref{epralintro}).
\end{itemize} 

On the other hand, a second advance attained in this work is the construction of an integral representation of the deceleration operator in two variables associated to $\overline{\ell}=(\ell_4,\ell_5,\ell_6,\ell_7)\in\N^4$ and $k_1,k_2$, formally defined by
$$\hat{\mathcal{D}}_{\overline{\ell},k_1,k_2}\left(\sum_{n_1,n_2\ge1}U_{n_1,n_2}t_1^{n_1}t_2^{n_2}\right)(h_1,h_2)=\sum_{n_1,n_2\ge1}U_{n_1,n_2}\frac{\Gamma\left(\frac{n_1}{k_1}\right)\Gamma\left(\frac{n_2}{k_2}\right)}{\Gamma\left(\frac{\ell_4n_1+\ell_6n_2}{k_1}\right)\Gamma\left(\frac{\ell_5n_1+\ell_7n_2}{k_2}\right)}h_1^{n_1}h_2^{n_2}.$$
This integral representation is written in the form of iterations of path integrals along segments together with Hankel-type integrals, and it leans on some kernel function $\mathbb{K}_{\overline{\ell},k_1,k_2}$. The details on such integral representation are given in Proposition~\ref{prop189}. The importance of such representation relies on the possibility to attain adequate bounds when applying this operator on holomorphic functions defined on some neighborhood of the origin in $\mathbb{C}^2$ (whose Taylor representation at the origin does not include linear nor constant terms). An integral representation of the one-dimensional version of such decelerator operator  can be found in~\cite{ma24}. Such operator is derived from the analytic deceleration operators initially described by J. \'Ecalle, which have been crucial in the study of multisummability, see~\cite{ba2}, Chapters 5 and 6.

Another remarkable feature hinges on the fact that our integral representation of
$\hat{\mathcal{D}}_{\bar{l},k_{1},k_{2}}$ involves the so-called Beta integral which classically appears in the construction of analytic solutions and connection problems for the Gauss hypergeometric differential equation. This integral is part of the Euler type integral representations and Riemann-Liouville transforms. Their study has become a very active domain of reseach related to rigidity problems for Fuchsian differential equations and linear Pfaffian systems by means of the notion of middle convolution in the framework of the Katz-Oshima theory. For a comprehensive introduction to this topic we refer to the textbook~\cite{haraokabook} by Y. Haraoka. For more recent developments on this subject we mention the works~\cite{adachi,adachikazuki,ebisuetal,haraoka26,oshima1,oshima2,oshima3}.

Section~\ref{sec3} is devoted to the statement of the main problem under study. The procedure rests on reducing the problem in two steps, one by the application of Fourier transform and the second by using formal Borel transforms. It is crucial to determine an integral representation of a two-variable version of the deceleration operator which allows to build the solution to an auxiliar equation in Section~\ref{sec4} following a fixed point argument in certain complex Banach spaces of functions. The main result, Theorem~\ref{teopral}, is enunciated in Section~\ref{secpral}. Two appendices conclude the paper: the first one recalling the definitions and main properties of certain function spaces and operators considered in the present work, and a second appendix describing the main definitions and properties about Majima's asymptotic expansions in several variables we use throughout the work.

\vspace{0.3cm}

\textbf{Notation:}
 
$\N$ stands for the set of non-negative integers, and $\N^{\star}=\N\setminus\{0\}$.

Given $\rho>0$, $D(0,\rho)$ stands for the disc centered at the origin, with radius $\rho$.

Let $(\mathbb{E},\left\|\cdot\right\|_{\mathbb{E}})$ be a complex Banach space. For every integer $n\ge1$, we denote by $\mathcal{O}(U,\mathbb{E})$ the set of holomorphic functions defined on some open set $\emptyset\neq U\subseteq \C^n$, with values in $\mathbb{E}$. We write $\mathcal{O}(U)$ in the case that $\mathbb{E}=\mathbb{C}$.

All sectors appearing in the present study are edged at the origin. Given a polysector $\boldsymbol{S}=S_1\times S_2\subseteq\C^2$, where $S_1$ and $S_2$ are sectors in $\C$, we say that $\boldsymbol{T}=T_1\times T_2$ is a (bounded and proper) subpolysector of $\boldsymbol{S}$ if $T_j$ is a bounded subsector of $S_j$ with $\overline{T_j}\setminus\{0\}\subseteq S_j$, for $j=1,2$. We write $\boldsymbol{T}\prec \boldsymbol{S}$. 

At some places, we use the notation $m_{k}(n):=\Gamma(n/k)$, for given $k,n>0$.

\section{Statement of the main problem}\label{sec3}

In this section, we concretize the problem under study in the present research.

Let $k_1,k_2\ge1$ and $\Delta_{1},\Delta_{2}\ge 2$ be natural numbers. We also fix positive real numbers $\alpha_{1},\alpha_{2}$ and $\tilde{c}\in\C^{\star}$. The set $\mathcal{A}\subset\N^8$ stands for a finite set of 8-tuples of the form $\underline{\ell}=(\ell_0,\ell_1,\ldots,\ell_7)$ with $\ell_4,\ell_7\ge1$, which satisfy the following properties:
\begin{itemize}
\item[(0)] If $\ell_0=\ell_2=0$ and $(\ell_4,\ldots,\ell_7)=(1,0,0,1)$, then $\ell_1,\ell_3\ge0$.
\item[(i)]  If $\ell_0\ge1$, $\ell_2=0$ and $(\ell_4,\ldots,\ell_7)= (1,0,0,1)$, then $\ell_3=0$.
\item[(ii)]  If $\ell_0=0$, $\ell_2\ge1$ and $(\ell_4,\ldots,\ell_7)= (1,0,0,1)$, then $\ell_1=0$.
\item[(iii)] If $\ell_0,\ell_2\ge1$ and $(\ell_4,\ldots,\ell_7)=(1,0,0,1)$, then $\ell_1,\ell_3\ge0$.
\item[(iv)] If $\underline{\ell}\in\mathcal{A}$ is such that $\ell_4,\ell_7\ge2$ and $\ell_5,\ell_6\ge1$, then:
\begin{itemize}
\item there exist $p_j,q_j>0$ with $1/p_j+1/q_j=1$, for $j=1,2$ such that
\begin{equation}\label{e115}
\Delta_1> \max\left\{\frac{\ell_4p_1}{\ell_4-1},\frac{\ell_6k_2p_2}{(\ell_7-1)k_1}\right\},\qquad \Delta_2> \max\left\{\frac{\ell_5k_1q_1}{(\ell_4-1)k_2},\frac{\ell_7q_2}{\ell_7-1}\right\}
\end{equation}
\item $k_1\ge \ell_4$ and $k_2\ge \ell_7$.
\end{itemize}
\item[(v)] If $\underline{\ell}\in\mathcal{A}$ is such that $\ell_4,\ell_7\ge2$ and $\ell_5,\ell_6\ge1$ with $\ell_0=\ell_2=0$ then
$$\ell_1,\ell_3\ge1,\qquad \ell_1\ell_5 k_1\ge k_2, \qquad \ell_3\ell_6k_2\ge k_1.$$
\item[(vi)] If $\underline{\ell}\in\mathcal{A}$ is such that $\ell_4,\ell_7\ge2$ and $\ell_5,\ell_6\ge1$ with $\ell_0\ge1$ and $\ell_2=0$ then
$$\frac{\ell_4\ell_0}{k_1}+\ell_4\ell_1\ge1,\quad \ell_6k_2\ell_3\ge k_1,\quad \frac{\ell_5}{k_2}(\ell_0+k_1\ell_1)\ge1,\quad \ell_3\ell_7\ge1.$$
\item[(vii)] If $\underline{\ell}\in\mathcal{A}$ is such that $\ell_4,\ell_7\ge2$ and $\ell_5,\ell_6\ge1$ with $\ell_0=0$ and $\ell_2\ge 1$, then
$$\frac{\ell_2\ell_7}{k_2}+\ell_3\ell_7\ge1,\quad \ell_5k_1\ell_1\ge k_2,\quad \frac{\ell_6}{k_1}(\ell_2+k_2\ell_3)\ge1,\quad \ell_4\ell_1\ge1.$$
\item[(viii)] If $\underline{\ell}\in\mathcal{A}$ is such that $\ell_4,\ell_7\ge2$ and $\ell_5,\ell_6\ge1$ with $\ell_0\ge 1$ and $\ell_2\ge 1$, then
$$\ell_0\ge k_1,\quad \ell_1\ge1,\quad \ell_2\ge k_2,\quad \ell_3\ge 1.$$
\end{itemize}

Let $Q_j(X),\tilde{Q}_j(X),R_{j}(X),R_{\underline{\ell}}(X)\in\C[X]$, for $j=1,2$ and all $\underline{\ell}\in\mathcal{A}$. We assume that $R_j(im)\neq 0$ for $j=1,2$ and all $m\in\R$, together with
\begin{equation}\label{e110}
\hbox{deg}(R_1)+\hbox{deg}(R_2)\ge \hbox{deg}(R_{\underline{\ell}}),\qquad  \hbox{deg}(R_j)=\hbox{deg}(Q_j),\qquad j=1,2.
\end{equation}
\begin{equation}\label{e112}
\hbox{deg}(R_1)+\hbox{deg}(R_2)\ge \max\{\hbox{deg}(\tilde{Q}_1),\hbox{deg}(\tilde{Q}_2)\}.
\end{equation}

We assume that for $j=1,2$, there exist $0<r_{Q_j,R_j,1}<r_{Q_j,R_j,2}$, $\nu_{Q_j,R_j}>0$, and $d_{Q_j,R_j}\in\R$ such that
\begin{equation}\label{e113}
\left\{\frac{Q_j(im)}{R_{j}(im)}:m\in\R\right\}\subseteq S_{Q_j,R_j},
\end{equation}
where 
\begin{equation}\label{e114}
S_{Q_j,R_j}=\left\{z\in\C^{\star}:r_{Q_j,R_j,1}\le |z|\le r_{Q_j,R_j,2},\quad |\hbox{arg}(z)-d_{Q_j,R_j}|\le\nu_{Q_j,R_j}\right\}.
\end{equation}

Later on in the work, constraints will be imposed on the location of the above sectorial domains $S_{Q_j,R_j}$ for $j=1,2$ (see Proposition~\ref{prop426}).

We assume $\beta>0$ and $\mu>1$ are real numbers such that
\begin{equation}\label{e142}
\mu\ge \hbox{deg}(R_{\underline{\ell}})+1,\quad \hbox{for all }\underline{\ell}\in\mathcal{A},\qquad \mu>\max\{\hbox{deg}(\tilde{Q}_1),\hbox{deg}(\tilde{Q}_2)\}+1.
\end{equation}

The main problem under study is determined by
\begin{multline}
\left[Q_1(\partial_z)-\cosh\left(\alpha_1(t_1^{k_1+1}\partial_{t_1})^{\Delta_1}\right)R_{1}(\partial_z)\right]\\
\times \left[Q_2(\partial_z)-\cosh\left(\alpha_2(t_2^{k_2+1}\partial_{t_2})^{\Delta_2}\right)R_{2}(\partial_z)\right]u(t_1,t_2,z)\\
=\sum_{\underline{\ell}=(\ell_0,\ldots,\ell_7)\in\mathcal{A}}a_{\underline{\ell}}(z)\left(t_1^{\ell_0}(t_1^{k_1+1}\partial_{t_1})^{\ell_1}t_2^{\ell_2}(t_2^{k_2+1}\partial_{t_2})^{\ell_3}R_{\underline{\ell}}(\partial_z)u\right)(t_1^{\ell_4}t_2^{\ell_5},t_1^{\ell_6}t_2^{\ell_7},z)\\
+\tilde{c}\tilde{Q}_1(\partial_z)u(t_1,t_2,z)\tilde{Q}_2(\partial_z)u(t_1,t_2,z)+f(t_1,t_2,z),\label{epral}
\end{multline}
under initial data $u(t_1,0,z)\equiv u(0,t_2,z)\equiv 0$.

In the previous expression, the infinite order differential operators $\cosh(\alpha_j(t_j^{k_j+1}\partial_{t_j})^{\Delta_j})$ for $j=1,2$ is defined by
\begin{equation}\label{e220}
\cosh(\alpha_j(t_j^{k_j+1}\partial_{t_j})^{\Delta_j})=\frac{1}{2}\left(\exp( \alpha_j(t_j^{k_j+1}\partial_{t_j})^{\Delta_j})+\exp(-\alpha_j(t_j^{k_j+1}\partial_{t_j})^{\Delta_j})\right),
\end{equation}
with
$$\exp(\pm\alpha_j(t_j^{k_j+1}\partial_{t_j})^{\Delta_j})=\sum_{p\ge0}\frac{(\pm\alpha_j)^p}{p!}(t_j^{k_j+1}\partial_{t_j})^{(\Delta_jp)},$$
and where $(t_j^{k_j+1}\partial_{t_j})^{(\Delta_jp)}$ stands for the $\Delta_jp$-th iterate of the differential operator $t_j^{k_j+1}\partial_{t_j}$.

For every $\underline{\ell}\in\mathcal{A}$, the coefficient $a_{\underline{\ell}}$ is constructed as follows. Let $m\mapsto A_{\underline{\ell}}(m)$ be in $E_{(\beta,\mu)}$ (see Section~\ref{secanexo1} for the definition of $E_{(\beta,\mu)}$). We denote $\mathfrak{A}_{\underline{\ell}}=\left\|A_{\underline{\ell}}\right\|_{(\beta,\mu)}$ so that the following estimates hold
$$|A_{\underline{\ell}}(m)|\le \mathfrak{A}_{\underline{\ell}} (1+|m|)^{-\mu}e^{-\beta|m|},$$
for all $m\in\mathbb{R}$. The function 
$$a_{\underline{\ell}}(z)=\mathcal{F}^{-1}(A_{\underline{\ell}})(z)=\frac{1}{(2\pi)^{1/2}}\int_{-\infty}^{+\infty}A_{\underline{\ell}}(m)e^{izm}dm$$
turns out to be a bounded holomorphic function on $H_{\beta'}=\{z\in\C:|\hbox{Im}(z)|<\beta'\}$, for every $0<\beta'<\beta$.

The forcing term $f$ is constructed as follows. Let $J\subseteq(\N^{\star})^2$ be a finite set. For every $\underline{j}=(j_1,j_2)\in J$, we consider a function $m\mapsto \mathcal{F}_{\underline{j}}(m)$ belonging to $E_{(\beta,\mu)}$ and put $\mathfrak{F}_{\underline{j}}=\left\|\mathcal{F}_{\underline{j}}\right\|_{(\beta,\mu)}$. We define $F_{\underline{j}}=\mathcal{F}^{-1}(\mathcal{F}_{\underline{j}})$, which turns out to be a bounded holomorphic function on $H_{\beta'}$, for all $0<\beta'<\beta$. We define
$$f(t_1,t_2,z)=\sum_{\underline{j}=(j_1,j_2)\in J}F_{\underline{j}}(z)\Gamma(j_1/k_1)\Gamma(j_2/k_2)t_1^{j_1}t_2^{j_2}.$$

\subsection{First reduction of the main problem}\label{sec21}

In this section, we clarify the strategy to construct a solution to (\ref{epral}) following several steps, describing the first of such steps. We search for solutions in the form of an inverse Fourier transform $u(t_1,t_2,z)=\mathcal{F}^{-1}(m\mapsto U(t_1,t_2,m))(z)$, for some expression $U(t_1,t_2,m)$. Regarding Section~\ref{secanexo1}, one has that $u(t_1,t_2,z)$ is a formal solution of (\ref{epral}) if $U(t_1,t_2,z)$ is a formal solution to 

\begin{multline}
\left[Q_1(im)-\cosh\left(\alpha_1(t_1^{k_1+1}\partial_{t_1})^{\Delta_1}\right)R_{1}(im)\right]\\
\times \left[Q_2(im)-\cosh\left(\alpha_2(t_2^{k_2+1}\partial_{t_2})^{\Delta_2}\right)R_{2}(im)\right]U(t_1,t_2,m)\\
=\sum_{\underline{\ell}=(\ell_0,\ldots,\ell_7)\in\mathcal{A}}\frac{1}{(2\pi)^{1/2}}\int_{-\infty}^{\infty}A_{\underline{\ell}}(m-m_1)\hfill\\
\times\left(t_1^{\ell_0}(t_1^{k_1+1}\partial_{t_1})^{\ell_1}t_2^{\ell_2}(t_2^{k_2+1}\partial_{t_2})^{\ell_3}U\right)(t_1^{\ell_4}t_2^{\ell_5},t_1^{\ell_6}t_2^{\ell_7},m_1)R_{\underline{\ell}}(im_1)dm_1\\
+\tilde{c}\frac{1}{(2\pi)^{1/2}}\int_{-\infty}^{\infty}\tilde{Q}_1(i(m-m_1))U(t_1,t_2,m-m_1)\tilde{Q}_2(im_1)U(t_1,t_2,m_1)dm_1\\
+\sum_{\underline{j}\in J}\mathcal{F}_{\underline{j}}(m)\Gamma(j_1/k_1)\Gamma(j_2/k_2)t_1^{j_1}t_2^{j_2},\label{eaux1}
\end{multline}
for $U(t_1,0,m)\equiv U(0,t_2,m)\equiv 0$.

We search for solutions to (\ref{eaux1}) in the form of a formal power series
\begin{equation}\label{e153}
\hat{U}(t_1,t_2,m)=\sum_{n_1,n_2\ge1}U_{n_1,n_2}(m)t_1^{n_1}t_2^{n_2},
\end{equation}
belonging to the space of formal power series in two complex time variables $E_{(\beta,\mu)}[[t_1,t_2]]$.

\subsection{Multivariate Mahler operators, formal two variables Borel transform and deceleration operator}

Before performing the second reduction of the problem, we describe the action of Mahler operators on a two variables Borel transform, and its relation with a deceleration operator. First, we define the formal $(m_{k_1},m_{k_2})-$Borel transform of a formal power series in two variables.

Let $(\mathbb{E},\left\|\cdot\right\|_{\mathbb{E}})$ be a complex Banach space, which is a Banach algebra with product denoted by $*$.

\begin{defin}\label{defi166}
Let $k_1,k_2$ be positive integers, and let $\hat{U}(t_1,t_2)\in t_1t_2\mathbb{E}[[t_1,t_2]]$ given by
\begin{equation}\label{e160}
\hat{U}(t_1,t_2)=\sum_{n_1,n_2\ge1}U_{n_1,n_2}t_1^{n_1}t_2^{n_2}.
\end{equation}
The formal $(m_{k_1},m_{k_2})-$Borel transform of $\hat{U}$ is defined by
$$\mathcal{B}_{(m_{k_1},m_{k_2})}(\hat{U})(\tau_1,\tau_2):=\sum_{n_1,n_2\ge1}U_{n_1,n_2}\frac{\tau_1^{n_1}}{\Gamma\left(\frac{n_1}{k_1}\right)}\frac{\tau_2^{n_2}}{\Gamma\left(\frac{n_2}{k_2}\right)}.$$
Let $\ell_j$ be positive integers for $4\le j\le 7$. The formal deceleration operator of $\hat{U}$, associated to $\overline{\ell}=(\ell_4,\ldots,\ell_7)$ and $k_1,k_2$ is defined by
$$\hat{\mathcal{D}}_{\overline{\ell},k_1,k_2}(\hat{U})(h_1,h_2):=\sum_{n_1,n_2\ge1}U_{n_1,n_2}\frac{\Gamma\left(\frac{n_1}{k_1}\right)\Gamma\left(\frac{n_2}{k_2}\right)}{\Gamma\left(\frac{\ell_4 n_1+\ell_6 n_2}{k_1}\right)\Gamma\left(\frac{\ell_5 n_1+\ell_7 n_2}{k_2}\right)}h_1^{n_1}h_2^{n_2}.$$
\end{defin}

The following properties of formal Borel transform can be adapted from their one-dimensional counterpart, which can be found in detail in Proposition 6,~\cite{lama}, so we omit their proof.

\begin{prop}\label{prop177}
Let $k_1,k_2,\ell$ be positive integers, and let $\hat{U}(t_1,t_2)\in t_1t_2\mathbb{E}[[t_1,t_2]]$. The following properties hold for all $j=1,2$:
\begin{multline*}
\mathcal{B}_{(m_{k_1},m_{k_2})}(t_j^{k_j+1}\partial_{t_j}\hat{U}(t_1,t_2))(\tau_1,\tau_2)=k_j\tau_j^{k_j}\mathcal{B}_{(m_{k_1},m_{k_2})}(\hat{U})(\tau_1,\tau_2),\\
\mathcal{B}_{(m_{k_1},m_{k_2})}(t_j^{\ell}\hat{U}(t_1,t_2))(\tau_1,\tau_2)\hfill\\
\hfill=\frac{\tau_j^{k_j}}{\Gamma\left(\frac{\ell}{k_j}\right)}\int_{0}^{\tau_j^{k_j}}(\tau_j^{k_j}-s_j)^{\frac{\ell}{k_j}-1}\mathcal{B}_{(m_{k_1},m_{k_2})}(\hat{U})(s_1^{1/k_1}\delta_{1,j}+\tau_1\delta_{2,j},s_2^{1/k_2}\delta_{2,j}+\tau_2\delta_{1,j})\frac{ds_j}{s_j}.
\end{multline*}
In the second of the previous expressions, $\delta_{i,j}$ stands for Kronecker delta. Let $\hat{U}_1(t_1,t_2)\in t_1t_2\mathbb{E}[[t_1,t_2]]$. Then, one has
\begin{multline*}
\mathcal{B}_{(m_{k_1},m_{k_2})}(\hat{U}(t_1,t_2)*\hat{U}_1(t_1,t_2))(\tau_1,\tau_2)\hfill\\
=\tau_1^{k_1}\tau_2^{k_2}\int_0^{\tau_1^{k_1}}\int_0^{\tau_2^{k_2}} \mathcal{B}_{(m_{k_1},m_{k_2})}(\hat{U})((\tau_1^{k_1}-s_1)^{1/k_1},(\tau_2^{k_2}-s_2)^{1/k_2})\\
\times\mathcal{B}_{(m_{k_1},m_{k_2})}(\hat{U}_1)(s_1^{1/k_1},s_2^{1/k_2})\frac{ds_1ds_2}{(\tau_1^{k_1}-s_1)s_1(\tau_2^{k_2}-s_2)s_2}.
\end{multline*}
\end{prop}

The following result is a direct consequence of the definitions of formal Borel transform and the formal deceleration operator.

\begin{lemma}\label{lema197}
Let $k_1,k_2$ be positive integers and let $\ell_j$ for $4\le j\le 7$ be nonnegative integers with $\ell_4,\ell_7\ge1$. We write $\overline{\ell}=(\ell_4,\ldots,\ell_7)$. For every $\hat{U}(t_1,t_2)\in t_1t_2\mathbb{E}[[t_1,t_2]]$ one has that
$$\hat{\mathcal{D}}_{\overline{\ell},k_1,k_2}\left(\mathcal{B}_{(m_{k_1},m_{k_2})}(\hat{U})\right)(\tau_1^{\ell_4}\tau_2^{\ell_5},\tau_1^{\ell_6}\tau_2^{\ell_7})=\mathcal{B}_{(m_{k_1},m_{k_2})}(\hat{V})(\tau_1,\tau_2),$$
where $\hat{V}(t_1,t_2):=\hat{U}(t_1^{\ell_4}t_2^{\ell_5},t_1^{\ell_6}t_2^{\ell_7})$.
\end{lemma}
\begin{proof}
Let $\hat{U}$ be as in (\ref{e160}). We observe that 
$$\hat{U}(t_1^{\ell_4}t_2^{\ell_5},t_1^{\ell_6}t_2^{\ell_7})=\sum_{n_1,n_2\ge1}U_{n_1,n_2}t_1^{\ell_4n_1+\ell_6n_2}t_2^{\ell_5n_1+\ell_7n_2},$$
which entails that
$$\mathcal{B}_{(m_{k_1},m_{k_2})}(\hat{V})(\tau_1,\tau_2)=\sum_{n_1,n_2\ge1}U_{n_1,n_2}\frac{\tau_1^{\ell_4n_1+\ell_6n_2}\tau_2^{\ell_5n_1+\ell_7n_2}}{\Gamma\left(\frac{\ell_4n_1+\ell_6n_2}{k_1}\right)\Gamma\left(\frac{\ell_5n_1+\ell_7n_2}{k_2}\right)}.$$
The previous expression coincides with $\hat{\mathcal{D}}_{\overline{\ell},k_1,k_2}\left(\mathcal{B}_{(m_{k_1},m_{k_2})}(\hat{U})\right)(\tau_1^{\ell_4}\tau_2^{\ell_5},\tau_1^{\ell_6}\tau_2^{\ell_7})$, taking into account the definition of Borel and deceleration operator.  
\end{proof}

\textbf{Remark:} Observe that $\overline{\ell}=(1,0,0,1)$ corresponds to the absence of multidimensional Mahler operators, $\hat{U}\equiv \hat{V}$.

In the next result, we provide an integral representation for the deceleration operator $\hat{\mathcal{D}}_{\overline{\ell},k_1,k_2}$ which will be crucial in the achievement of the formal solution to the main problem in terms of the corresponding one associated to a second auxiliary equation.

\begin{prop}\label{prop189}
Let $(\mathbb{E},\left\|\cdot\right\|_{\mathbb{E}})$ be a complex Banach space. We also fix integers $k_1,k_2\ge 1$ and a tuple $\overline{\ell}=(\ell_4,\ldots,\ell_7)$ of nonnegative integers with $\ell_4,\ell_7\ge2$, and $\ell_5,\ell_6\ge1$. Let $\hat{f}(\tau_1,\tau_2)\in\tau_1\tau_2\mathbb{E}[[\tau_1,\tau_2]]$, which is assumed to be convergent on some neighborhood of the origin in $\C^2$, $D(0,\rho)\times D(0,\rho)$ for some $0<\rho<1$.

Let $j=1,2$. Given $h_j\in\C^{\star}$, we define three types of paths. A first path is determined by the choice of $\gamma_{h_j}\in\R$ such that there exists $\delta_j>0$ with
$$\cos(k_j(\gamma_{h_j}-\hbox{arg}(h_j)))>\delta_j.$$ 
We write $L_{\gamma_{h_j}}=[0,\infty)e^{i\gamma_{h_j}}=\{re^{i\gamma_{h_j}}:r\ge0\}$. A second path to be considered is a Hankel contour. For $j=1$ (resp. $j=2$), then $\ell=\ell_4$ (resp. $\ell=\ell_7$), the closed Hankel path, denoted by $\gamma_{k_j/\ell,h_j}$ consists of the concatenation of:
\begin{itemize}
\item[-] the oriented segment $[0,\frac{\rho}{2}e^{i(\gamma_{h_j}+\frac{\pi}{2k_j}\ell+\frac{\delta'}{2})}]$, say $\gamma_{k_j/\ell,h_j,1}$,
\item[-] the oriented arc $\{\frac{\rho}{2}e^{i\theta}:\theta\in[\gamma_{h_j}+\frac{\pi}{2k_j}\ell+\frac{\delta'}{2},\gamma_{h_j}-\frac{\pi}{2k_j}\ell-\frac{\delta'}{2}]$, say $\gamma_{k_j/\ell,h_j,3}$, 
\item[-] the oriented segment $[\frac{\rho}{2}e^{i(\gamma_{h_j}-\frac{\pi}{2k_j}\ell-\frac{\delta'}{2})},0]$, say $\gamma_{k_j/\ell,h_j,2}$,
\end{itemize}
where $\delta'$ is a small positive number. A third type of paths is defined as follows. For $j=1$ (resp. $j=2$), then $\ell=\ell_6$ (resp. $\ell=\ell_5$), the closed Hankel path, denoted by $\gamma_{k_j/\ell}$ consists of the concatenation of
\begin{itemize}
\item[-] the oriented segment $[0,\frac{\rho}{2}e^{i(\frac{\pi}{2k_j}\ell+\frac{\delta'}{2})}]$, say $\gamma_{k_j/\ell,1}$,
\item[-] the oriented arc of circle $\{\frac{\rho}{2}e^{i\theta}:\theta\in[\frac{\pi}{2k_j}\ell+\frac{\delta'}{2},-\frac{\pi}{2k_j}\ell-\frac{\delta'}{2}]$, say $\gamma_{k_j/\ell,3}$, 
\item[-] the oriented segment $[\frac{\rho}{2}e^{i(-\frac{\pi}{2k_j}\ell-\frac{\delta'}{2})},0]$, say $\gamma_{k_j/\ell,2}$,
\end{itemize}
where $\delta'$ is a small positive number. 

In this situation, the formal deceleration operator $\hat{\mathcal{D}}_{\overline{\ell},k_1,k_2}(\hat{f})$ has a multiple integral representation in the form
\begin{multline}
\hat{\mathcal{D}}_{\overline{\ell},k_1,k_2}(\hat{f})(h_1,h_2)\\
=\int_0^1\int_0^1\int_{\gamma_{\frac{k_1}{\ell_4},h_1}}\int_{\gamma_{\frac{k_2}{\ell_7},h_2}}\int_{\gamma_{\frac{{k_2}}{\ell_5}}}\int_{\gamma_{\frac{k_1}{\ell_6}}}\hat{f}((1-t_1)^{\frac{\ell_4}{k_1}}(1-t_2)^{\frac{\ell_5}{k_2}}\xi_1\xi_3,t_1^{\frac{\ell_6}{k_1}}t_2^{\frac{\ell_7}{k_2}}\xi_2\xi_4)\\
\times \mathbb{K}_{\overline{\ell},k_1,k_2}(\boldsymbol{t},\boldsymbol{\xi},\boldsymbol{h})dt_1dt_2d\xi_1d\xi_2d\xi_3d\xi_4,\label{e204}
\end{multline}
for some kernel function $\mathbb{K}_{\overline{\ell},k_1,k_2}(\boldsymbol{t},\boldsymbol{\xi},\boldsymbol{h})=\mathbb{K}_{\overline{\ell},k_1,k_2}(t_1,t_2,\xi_1,\xi_2,\xi_3,\xi_4,h_1,h_2)$.
\end{prop}
\begin{proof}

Let us write 
$$\hat{f}(\tau_1,\tau_2)=\sum_{n_1,n_2\ge 1}f_{n_1,n_2}\tau_1^{n_1}\tau_2^{n_2}\in\tau_1\tau_2\mathbb{E}[[\tau_1,\tau_2]].$$
Then, one has that
$$\hat{\mathcal{D}}_{\overline{\ell},k_1,k_2}(\hat{f})(h_1,h_2)=\sum_{n_1,n_2\ge1}\frac{\Gamma\left(\frac{n_1}{k_1}\right)\Gamma\left(\frac{n_2}{k_2}\right)}{\Gamma\left(\frac{\ell_4 n_1+\ell_6 n_2}{k_1}\right)\Gamma\left(\frac{\ell_5 n_1+\ell_7 n_2}{k_2}\right)}f_{n_1,n_2}h_1^{n_1}h_2^{n_2}.$$
From the definition of Gamma function, one has the following integral representations (see~\cite{ma24} for a reference on this representation)
\begin{equation}\label{e211}
\Gamma\left(\frac{n_j}{k_j}\right)h_j^{n_j}=k_j\int_{L_{\gamma_{h_j}}}u_j^{n_j}\exp\left(-\left(\frac{u_j}{h_j}\right)^{k_j}\right)\frac{d u_j}{u_j},
\end{equation}
for $j=1,2$.

We split the expression
$$\frac{\Gamma\left(\frac{n_1}{k_1}\right)\Gamma\left(\frac{n_2}{k_2}\right)}{\Gamma\left(\frac{\ell_4 n_1+\ell_6 n_2}{k_1}\right)\Gamma\left(\frac{\ell_5 n_1+\ell_7 n_2}{k_2}\right)}h_1^{n_1}h_2^{n_2}$$
and provide integral representations for each part in the splitting.

\noindent\textbf{Step 1:} We write 
\begin{equation}\label{e215}
\frac{1}{\Gamma\left(\frac{\ell_4n_1+\ell_6n_2}{k_1}\right)}=\frac{\Gamma\left(\frac{\ell_4n_1}{k_1}\right)\Gamma\left(\frac{\ell_6n_2}{k_1}\right)}{\Gamma\left(\frac{\ell_4n_1+\ell_6n_2}{k_1}\right)}\frac{1}{\Gamma\left(\frac{\ell_4n_1}{k_1}\right)}\frac{1}{\Gamma\left(\frac{\ell_6n_2}{k_1}\right)}.
\end{equation}
It is known (see~\cite{ba2}) that for every $\alpha,\beta\in\C$ with $\hbox{Re}(\alpha),\hbox{Re}(\beta)>0$, Beta function has the following integral representation
$$\frac{\Gamma(\alpha)\Gamma(\beta)}{\Gamma(\alpha+\beta)}=B(\alpha,\beta)=\int_0^1(1-t)^{\alpha-1}t^{\beta-1}dt.$$
This yields
\begin{equation}\label{e221}
\frac{\Gamma\left(\frac{\ell_4n_1}{k_1}\right)\Gamma\left(\frac{\ell_6n_2}{k_1}\right)}{\Gamma\left(\frac{\ell_4n_1+\ell_6n_2}{k_1}\right)}=\int_0^1(1-t_1)^{\frac{\ell_4n_1}{k_1}-1}t_1^{\frac{\ell_6n_2}{k_1}-1}dt_1.
\end{equation}
Second, we consider the following integral representation via a Hankel path:
\begin{equation}\label{e225}
\frac{u_1^{n_1}}{\Gamma\left(\frac{\ell_4 n_1}{k_1}\right)}=-\frac{k_1/\ell_4}{2\pi i}\int_{\gamma_{\frac{k_1}{\ell_4},h_1}}\xi_1^{n_1}\exp\left(\left(\frac{u_1}{\xi_1}\right)^{\frac{k_1}{\ell_4}}\right)\frac{u_1^{k_1/\ell_4}}{\xi_1^{\frac{k_1}{\ell_4}+1}}d\xi_1,
\end{equation}
for $u_1\in L_{\gamma_{h_1}}$. We refer to~\cite{ma24}, Proposition 6, for an explanation of this representation.

Taking into account (\ref{e225}) together with (\ref{e211}) for $j=1$ one has
\begin{multline}
\frac{\Gamma\left(\frac{n_1}{k_1}\right)}{\Gamma\left(\frac{\ell_4n_1}{k_1}\right)}h_1^{n_1}\\
=k_1\int_{L_{\gamma_{h_1}}}\left[-\frac{k_1/\ell_4}{2\pi i}\int_{\gamma_{\frac{k_1}{\ell_4},h_1}}\xi_1^{n_1}\exp\left(\left(\frac{u_1}{\xi_1}\right)^{\frac{k_1}{\ell_4}}\right)\frac{u_1^{k_1/\ell_4}}{\xi_1^{\frac{k_1}{\ell_4}+1}}d\xi_1 \right]\exp\left(-\left(\frac{u_1}{h_1}\right)^{k_1}\right)\frac{du_1}{u_1}\\
=-\frac{k_1^2\ell_4}{2\pi i}\int_{\gamma_{\frac{k_1}{\ell_4},h_1}}\xi_1^{n_1}\left[\frac{1}{\xi_1^{\frac{k_1}{\ell_4}+1}}\int_{L_{\gamma_{h_1}}}u_1^{\frac{k_1}{\ell_4}}\exp\left(\left(\frac{u_1}{\xi_1}\right)^{\frac{k_1}{\ell_4}}-\left(\frac{u_1}{h_1}\right)^{k_1}\right) \right]\frac{du_1}{u_1}d\xi_1,\label{e238}
\end{multline}
the last step being a consequence of Fubini's theorem. The above integral along $L_{\gamma_{h_1}}$ is convergent for given $h_1\in\C^{\star}$ owing to the conditions $\cos(k_1(\gamma_{h_1}-\hbox{arg}(h_1)))>\delta_1$ and $\ell_4\ge2$. The following integral representation along a Hankel path is also considered:
\begin{equation}\label{e237}
\frac{1}{\Gamma\left(\frac{\ell_6n_2}{k_1}\right)}=-\frac{k_1/\ell_6}{2\pi i}\int_{\gamma_{\frac{k_1}{\ell_6}}}\xi_4^{n_2}\exp\left(\left(\frac{1}{\xi_4}\right)^{\frac{k_1}{\ell_6}}\right)\frac{1}{\xi_4^{\frac{k_1}{\ell_6}+1}}d\xi_4.
\end{equation}

\noindent\textbf{Step 2:} An analogous procedure can be followed by writting
\begin{equation}\label{e215b}
\frac{1}{\Gamma\left(\frac{\ell_5n_1+\ell_7n_2}{k_2}\right)}=\frac{\Gamma\left(\frac{\ell_5n_1}{k_2}\right)\Gamma\left(\frac{\ell_7n_2}{k_2}\right)}{\Gamma\left(\frac{\ell_5n_1+\ell_7n_2}{k_2}\right)}\frac{1}{\Gamma\left(\frac{\ell_5n_1}{k_2}\right)}\frac{1}{\Gamma\left(\frac{\ell_7n_2}{k_2}\right)}.
\end{equation}
The first term in the previous product is written in integral form as in (\ref{e221}), whereas the integral representation of $u_2^{n_2}/\Gamma\left(\frac{\ell_7n_2}{k_2}\right)$ is analogous to that of (\ref{e225})
$$\frac{u_2^{n_2}}{\Gamma\left(\frac{\ell_7 n_2}{k_2}\right)}=-\frac{k_2/\ell_7}{2\pi i}\int_{\gamma_{\frac{k_2}{\ell_7},h_2}}\xi_2^{n_2}\exp\left(\left(\frac{u_2}{\xi_2}\right)^{\frac{k_2}{\ell_7}}\right)\frac{u_2^{k_2/\ell_7}}{\xi_2^{\frac{k_2}{\ell_7}+1}}d\xi_2,$$
where $u_2\in L_{\gamma_{h_2}}$, leading to the integral representation of $\frac{\Gamma(n_2/k_2)}{\Gamma(\ell_7n_2/k_2)}h_2^{n_2}$ in the shape of (\ref{e238}). Namely,
$$\frac{\Gamma\left(\frac{n_2}{k_2}\right)}{\Gamma\left(\frac{\ell_7n_2}{k_2}\right)}h_2^{n_2}
=-\frac{k_2^2\ell_7}{2\pi i}\int_{\gamma_{\frac{k_2}{\ell_7},h_2}}\xi_2^{n_2}\left[\frac{1}{\xi_2^{\frac{k_2}{\ell_7}+1}}\int_{L_{\gamma_{h_2}}}u_2^{\frac{k_2}{\ell_7}}\exp\left(\left(\frac{u_2}{\xi_2}\right)^{\frac{k_2}{\ell_7}}-\left(\frac{u_2}{h_2}\right)^{k_2}\right) \right]\frac{du_2}{u_2}d\xi_2.$$
For prescribed $h_2\in\C^{\star}$, the integral along $L_{\gamma_{h_2}}$ is convergent according to the restriction $\cos(k_2(\gamma_{h_2}-\hbox{arg}(h_2)))>\delta_2$ imposed and $\ell_7\ge2$.  An analogous integral representation to that in (\ref{e237}) is provided for $1/\Gamma(\ell_5n_1/k_2)$. More precisely, one has
\begin{equation}\label{e309}
\frac{1}{\Gamma\left(\frac{\ell_5n_1}{k_2}\right)}=-\frac{k_2/\ell_5}{2\pi i}\int_{\gamma_{\frac{k_2}{\ell_5}}}\xi_3^{n_1}\exp\left(\left(\frac{1}{\xi_3}\right)^{\frac{k_2}{\ell_5}}\right)\frac{1}{\xi_3^{\frac{k_2}{\ell_5}+1}}d\xi_3.
\end{equation}

\noindent\textbf{Step 3:} Putting all together the integral representations obtained in Step 1 and Step 2, one derives the following integral representation:
\begin{multline}
\frac{\Gamma\left(\frac{n_1}{k_1}\right)\Gamma\left(\frac{n_2}{k_2}\right)}{\Gamma\left(\frac{\ell_4 n_1+\ell_6 n_2}{k_1}\right)\Gamma\left(\frac{\ell_5 n_1+\ell_7 n_2}{k_2}\right)}h_1^{n_1}h_2^{n_2}\\
=\int_0^1\int_0^1\int_{\gamma_{\frac{k_1}{\ell_4},h_1}}\int_{\gamma_{\frac{k_2}{\ell_7},h_2}}\int_{\gamma_{\frac{k_2}{\ell_5}}}\int_{\gamma_{\frac{k_1}{\ell_6}}}
\left[(1-t_1)^{\frac{\ell_4}{k_1}}(1-t_2)^{\frac{\ell_5}{k_2}}\xi_1\xi_3\right]^{n_1}\\
\times\left[t_1^{\frac{\ell_6}{k_1}}t_2^{\frac{\ell_7}{k_2}}\xi_2\xi_4\right]^{n_2}\mathbb{K}_{\overline{\ell},k_1,k_2}(\boldsymbol{t},\boldsymbol{\xi},\boldsymbol{h})dt_1dt_2d\xi_1d\xi_2d\xi_3d\xi_4,\label{e270}
\end{multline}
where
\begin{multline}
\mathbb{K}_{\overline{\ell},k_1,k_2}(\boldsymbol{t},\boldsymbol{\xi},\boldsymbol{h})=\mathbb{K}_{\overline{\ell},k_1,k_2}(t_1,t_2,\xi_1,\xi_2,\xi_3,\xi_4,h_1,h_2)\\
=\frac{1}{t_1(1-t_1)}\frac{1}{t_2(1-t_2)}\left[-\frac{k_1^2/\ell_4}{2\pi i}\frac{1}{\xi_1^{\frac{k_1}{\ell_4}+1}}\left(\int_{L_{\gamma_{h_1}}}u_1^{\frac{k_1}{\ell_4}}\exp\left(\left(\frac{u_1}{\xi_1}\right)^{\frac{k_1}{\ell_4}}-\left(\frac{u_1}{h_1}\right)^{k_1}\right)\frac{du_1}{u_1}\right)\right]\\
\times \left[-\frac{k_2^2/\ell_7}{2\pi i}\frac{1}{\xi_2^{\frac{k_2}{\ell_7}+1}}\left(\int_{L_{\gamma_{h_2}}}u_2^{\frac{k_2}{\ell_7}}\exp\left(\left(\frac{u_2}{\xi_2}\right)^{\frac{k_2}{\ell_7}}-\left(\frac{u_2}{h_2}\right)^{k_2}\right)\frac{du_2}{u_2}\right)\right]\\
\hfill\times \left[-\frac{k_2/\ell_5}{2\pi i}\exp\left(\left(\frac{1}{\xi_3}\right)^{\frac{k_2}{\ell_5}}\right)\frac{1}{\xi_3^{\frac{k_2}{\ell_5}+1}}\right] \left[-\frac{k_1/\ell_6}{2\pi i}\exp\left(\left(\frac{1}{\xi_4}\right)^{\frac{k_1}{\ell_6}}\right)\frac{1}{\xi_4^{\frac{k_1}{\ell_6}+1}}\right]\\
=-\frac{k_1^2k_2^2}{4\pi^2\ell_4\ell_7}\frac{1}{t_1(1-t_1)}\frac{1}{t_2(1-t_2)}\mathbb{D}_{k_1,\frac{k_1}{\ell_4}}(\xi_1,h_1)\mathbb{D}_{k_2,\frac{k_2}{\ell_7}}(\xi_2,h_2)\hfill\\
\times 
\left[-\frac{k_2/\ell_5}{2\pi i}\exp\left(\left(\frac{1}{\xi_3}\right)^{\frac{k_2}{\ell_5}}\right)\frac{1}{\xi_3^{\frac{k_2}{\ell_5}+1}}\right] \left[-\frac{k_1/\ell_6}{2\pi i}\exp\left(\left(\frac{1}{\xi_4}\right)^{\frac{k_1}{\ell_6}}\right)\frac{1}{\xi_4^{\frac{k_1}{\ell_6}+1}}\right]
,\label{e275}
\end{multline}
where
$$\mathbb{D}_{k_1,\frac{k_1}{\ell_4}}(\xi_1,h_1)=\frac{1}{\xi_1^{\frac{k_1}{\ell_4}+1}}\int_{L_{\gamma_{h_1}}}u_1^{\frac{k_1}{\ell_4}}\exp\left(\left(\frac{u_1}{\xi_1}\right)^{\frac{k_1}{\ell_4}}-\left(\frac{u_1}{h_1}\right)^{k_1}\right)\frac{du_1}{u_1},$$
and
$$\mathbb{D}_{k_2,\frac{k_2}{\ell_7}}(\xi_2,h_2)=\frac{1}{\xi_2^{\frac{k_2}{\ell_7}+1}}\int_{L_{\gamma_{h_2}}}u_2^{\frac{k_2}{\ell_7}}\exp\left(\left(\frac{u_2}{\xi_2}\right)^{\frac{k_2}{\ell_7}}-\left(\frac{u_2}{h_2}\right)^{k_2}\right)\frac{du_2}{u_2}.$$

The formal sum for $n_1,n_2\ge1$ in (\ref{e270}) leads us to the integral representation of the deceleration operator in (\ref{e204}).
\end{proof}

\subsection{Second reduction of the problem}\label{sec23}

In this section, we take a step forward in the reduction of the main problem, departing from the auxiliary problem (\ref{eaux1}). We write
\begin{equation}\label{e327}
\hat{W}(\tau_1,\tau_2,m)=\mathcal{B}_{(m_{k_1},m_{k_2})}((t_1,t_2)\mapsto \hat{U}(t_1,t_2,m))(\tau_1,\tau_2),
\end{equation}
for the formal $(m_{k_1},m_{k_2})-$Borel transform of $\hat{U}$, given in (\ref{e153}), and determine an integral equation satisfied by (\ref{e327}). We assume that 
\begin{equation}\label{e349}
\hat{W}\in E_{(\beta,\mu)}\{\tau_1,\tau_2\}.
\end{equation}
Let us write $0<\rho<1$ such that the series $\hat{W}$ is convergent in $D(0,\rho)^2$.

We define for all $(\tau_1,\tau_2)\in\C^2$ and any $m\in\R$ the functions
$$P_{m,j}(\tau_j)=Q_{j}(im)-\cosh\left(\alpha_j(k_j\tau_j^{k_j})^{\Delta_j}\right)R_j(im),\qquad j=1,2.$$

In view of Lemma~\ref{lema197}, together with Proposition~\ref{prop189}, and the integral representation of the formal deceleration operator obtained in (\ref{e204}) valid due to assumption (\ref{e349}), one arrives at the second reduction of the main problem.

\begin{prop}\label{prop367}
The formal power series $\hat{U}(t_1,t_2,m)$ defined by (\ref{e160}) is a solution of (\ref{eaux1}) under initial data given by $\hat{U}(t_1,0,m)\equiv \hat{U}(0,t_2,m)\equiv 0$ if the convergent power series $\hat{W}(\tau_1,\tau_2,m)$ defined by (\ref{e327}) satisfies the following integral equation, where we have denoted
$$\mathcal{A}_1=\mathcal{A}_1(t_1,t_2,\xi_1,\xi_3):=(1-t_1)^{\frac{\ell_4}{k_1}}(1-t_2)^{\frac{\ell_5}{k_2}}\xi_1\xi_3,$$
$$\mathcal{A}_2=\mathcal{A}_2(t_1,t_2,\xi_2,\xi_4):=t_1^{\frac{\ell_6}{k_1}}t_2^{\frac{\ell_7}{k_2}}\xi_2\xi_4,$$
\begin{multline}
P_{m,1}(\tau_1)P_{m,2}(\tau_2)\hat{W}(\tau_1,\tau_2,m)\\ \label{eaux2}
=\sum_{\substack{\underline{\ell}=(\ell_0,\ldots,\ell_7)\in\mathcal{A}\\ (\ell_4,\ell_5,\ell_6,\ell_7)=(1,0,0,1)\\ \ell_0=0,\ell_2=0}}\frac{1}{(2\pi)^{1/2}}\int_{-\infty}^{\infty}A_{\underline{\ell}}(m-m_1)(k_1\tau_1^{k_1})^{\ell_1}(k_2\tau_2^{k_2})^{\ell_3}\hat{W}(\tau_1,\tau_2,m_1)R_{\underline{\ell}}(im_1)dm_1\\
+\sum_{\substack{\underline{\ell}=(\ell_0,\ldots,\ell_7)\in\mathcal{A}\\ (\ell_4,\ell_5,\ell_6,\ell_7)=(1,0,0,1)\\ \ell_0\ge 1,\ell_2=0}}\frac{1}{(2\pi)^{1/2}}\int_{-\infty}^{\infty}A_{\underline{\ell}}(m-m_1) C_{k_1,\ell_0,\ell_1}(\hat{W})(\tau_1,\tau_2,m_1)R_{\underline{\ell}}(im_1)dm_1\\
+\sum_{\substack{\underline{\ell}=(\ell_0,\ldots,\ell_7)\in\mathcal{A}\\ (\ell_4,\ell_5,\ell_6,\ell_7)=(1,0,0,1)\\ \ell_0=0,\ell_2\ge 1}}\frac{1}{(2\pi)^{1/2}}\int_{-\infty}^{\infty}A_{\underline{\ell}}(m-m_1) C_{k_2,\ell_2,\ell_3}(\hat{W})(\tau_1,\tau_2,m_1)R_{\underline{\ell}}(im_1)dm_1\\
+\sum_{\substack{\underline{\ell}=(\ell_0,\ldots,\ell_7)\in\mathcal{A}\\ (\ell_4,\ell_5,\ell_6,\ell_7)=(1,0,0,1)\\ \ell_0\ge 1,\ell_2\ge 1}}\frac{1}{(2\pi)^{1/2}}\int_{-\infty}^{\infty}A_{\underline{\ell}}(m-m_1) C_{k_1,k_2,\ell_0,\ell_2,\ell_1,\ell_3}(\hat{W})(\tau_1,\tau_2,m_1)R_{\underline{\ell}}(im_1)dm_1\\
+\sum_{\substack{\underline{\ell}=(\ell_0,\ldots,\ell_7)\in\mathcal{A}\\ \ell_4,\ell_7\ge2, \ell_5,\ell_6\ge1\\ \ell_0=0,\ell_2=0}}\frac{1}{(2\pi)^{1/2}}\int_{-\infty}^{\infty}A_{\underline{\ell}}(m-m_1) 
\int_0^1\int_0^1\int_{\gamma_{\frac{k_1}{\ell_4},\tau_1^{\ell_4}\tau_2^{\ell_5}}}
\int_{\gamma_{\frac{k_2}{\ell_7},\tau_1^{\ell_6}\tau_2^{\ell_7}}}
\int_{\gamma_{\frac{k_2}{\ell_5}}}
\int_{\gamma_{\frac{k_1}{\ell_6}}}\\
(k_1\mathcal{A}_1^{k_1})^{\ell_1}(k_2\mathcal{A}_2^{k_2})^{\ell_3}\hat{W}(\mathcal{A}_1,\mathcal{A}_2,m_1)\mathbb{K}_{\overline{\ell},k_1,k_2}(\boldsymbol{t},\boldsymbol{\xi},\tau_1^{\ell_4}\tau_2^{\ell_5},\tau_1^{\ell_6}\tau_2^{\ell_7})dt_1dt_2d\xi_1d\xi_2d\xi_3d\xi_4 R_{\underline{\ell}}(im_1)dm_1\\
+\sum_{\substack{\underline{\ell}=(\ell_0,\ldots,\ell_7)\in\mathcal{A}\\ \ell_4,\ell_7\ge2, \ell_5,\ell_6\ge1\\ \ell_0\ge1,\ell_2=0}}\frac{1}{(2\pi)^{1/2}}\int_{-\infty}^{\infty}A_{\underline{\ell}}(m-m_1) 
\int_0^1\int_0^1\int_{\gamma_{\frac{k_1}{\ell_4},\tau_1^{\ell_4}\tau_2^{\ell_5}}}
\int_{\gamma_{\frac{k_2}{\ell_7},\tau_1^{\ell_6}\tau_2^{\ell_7}}}
\int_{\gamma_{\frac{k_2}{\ell_5}}}
\int_{\gamma_{\frac{k_1}{\ell_6}}}\\
(k_2\mathcal{A}_2^{k_2})^{\ell_3}C_{k_1,\ell_0,\ell_1}(\hat{W})(\mathcal{A}_1,\mathcal{A}_2,m_1)\mathbb{K}_{\overline{\ell},k_1,k_2}(\boldsymbol{t},\boldsymbol{\xi},\tau_1^{\ell_4}\tau_2^{\ell_5},\tau_1^{\ell_6}\tau_2^{\ell_7})dt_1dt_2d\xi_1d\xi_2d\xi_3d\xi_4 R_{\underline{\ell}}(im_1)dm_1\\
+\sum_{\substack{\underline{\ell}=(\ell_0,\ldots,\ell_7)\in\mathcal{A}\\ \ell_4,\ell_7\ge2, \ell_5,\ell_6\ge1\\ \ell_0=0,\ell_2\ge1}}\frac{1}{(2\pi)^{1/2}}\int_{-\infty}^{\infty}A_{\underline{\ell}}(m-m_1) 
\int_0^1\int_0^1
\int_{\gamma_{\frac{k_1}{\ell_4},\tau_1^{\ell_4}\tau_2^{\ell_5}}}
\int_{\gamma_{\frac{k_2}{\ell_7},\tau_1^{\ell_6}\tau_2^{\ell_7}}}
\int_{\gamma_{\frac{k_2}{\ell_5}}}
\int_{\gamma_{\frac{k_1}{\ell_6}}}\\
(k_1\mathcal{A}_1^{k_1})^{\ell_1}C_{k_2,\ell_2,\ell_3}(\hat{W})(\mathcal{A}_1,\mathcal{A}_2,m_1)\mathbb{K}_{\overline{\ell},k_1,k_2}(\boldsymbol{t},\boldsymbol{\xi},\tau_1^{\ell_4}\tau_2^{\ell_5},\tau_1^{\ell_6}\tau_2^{\ell_7})dt_1dt_2d\xi_1d\xi_2d\xi_3d\xi_4 R_{\underline{\ell}}(im_1)dm_1\\
+\sum_{\substack{\underline{\ell}=(\ell_0,\ldots,\ell_7)\in\mathcal{A}\\ \ell_4,\ell_7\ge2, \ell_5,\ell_6\ge1\\ \ell_0\ge1,\ell_2\ge1}}\frac{1}{(2\pi)^{1/2}}\int_{-\infty}^{\infty}A_{\underline{\ell}}(m-m_1) 
\int_0^1\int_0^1\int_{\gamma_{\frac{k_1}{\ell_4},\tau_1^{\ell_4}\tau_2^{\ell_5}}}
\int_{\gamma_{\frac{k_2}{\ell_7},\tau_1^{\ell_6}\tau_2^{\ell_7}}}
\int_{\gamma_{\frac{k_2}{\ell_5}}}
\int_{\gamma_{\frac{k_1}{\ell_6}}}\\
C_{k_1,k_2,\ell_0,\ell_2,\ell_1,\ell_3}(\hat{W})(\mathcal{A}_1,\mathcal{A}_2,m_1)\mathbb{K}_{\overline{\ell},k_1,k_2}(\boldsymbol{t},\boldsymbol{\xi},\tau_1^{\ell_4}\tau_2^{\ell_5},\tau_1^{\ell_6}\tau_2^{\ell_7})dt_1dt_2d\xi_1d\xi_2d\xi_3d\xi_4 R_{\underline{\ell}}(im_1)dm_1
\end{multline}
\begin{multline*}
+\tilde{c}\frac{1}{(2\pi)^{1/2}}\int_{-\infty}^{\infty}\left(\tau_1^{k_1}\tau_2^{k_2}\int_0^{\tau_1^{k_1}}\int_0^{\tau_2^{k_2}} \hat{W}((\tau_1^{k_1}-s_1)^{1/k_1},(\tau_2^{k_2}-s_2)^{1/k_2},m-m_1)\tilde{Q}_1(i(m-m_1))\right.\\
\left.\times\hat{W}(s_1^{1/k_1},s_2^{1/k_2},m_1)\tilde{Q}_2(im_1)\frac{ds_1ds_2}{(\tau_1^{k_1}-s_1)s_1(\tau_2^{k_2}-s_2)s_2}\right)dm_1\\
+\sum_{\underline{j}=(j_1,j_2)\in J}\mathcal{F}_{\underline{j}}(m)\tau_1^{j_1}\tau_2^{j_2},
\end{multline*}
where 
$$C_{k_1,\ell_0,\ell_1}(\hat{W})(\tau_1,\tau_2,m)=\frac{\tau_1^{k_1}}{\Gamma\left(\frac{\ell_0}{k_1}\right)}\int_{0}^{\tau_1^{k_1}}(\tau_1^{k_1}-s_1)^{\frac{\ell_0}{k_1}-1}(k_1s_1)^{\ell_1}\hat{W}(s_1^{1/k_1},\tau_2,m)\frac{ds_1}{s_1},$$
$$C_{k_2,\ell_2,\ell_3}(\hat{W})(\tau_1,\tau_2,m)=\frac{\tau_2^{k_2}}{\Gamma\left(\frac{\ell_2}{k_2}\right)}\int_{0}^{\tau_2^{k_2}}(\tau_2^{k_2}-s_2)^{\frac{\ell_2}{k_2}-1}(k_2s_2)^{\ell_3}\hat{W}(\tau_1,s_2^{1/k_2},m)\frac{ds_2}{s_2},$$
and
\begin{multline*}
C_{k_1,k_2,\ell_0,\ell_2,\ell_1,\ell_3}(\hat{W})(\tau_1,\tau_2,m)\\
=\frac{\tau_1^{k_1}\tau_2^{k_2}}{\Gamma\left(\frac{\ell_0}{k_1}\right)\Gamma\left(\frac{\ell_2}{k_2}\right)}\int_{0}^{\tau_1^{k_1}}\int_{0}^{\tau_2^{k_2}}(\tau_1^{k_1}-s_1)^{\frac{\ell_0}{k_1}-1}(\tau_2^{k_2}-s_2)^{\frac{\ell_2}{k_2}-1}(k_1s_1)^{\ell_1}(k_2s_2)^{\ell_3}\hat{W}(s_1^{1/k_1},s_2^{1/k_2},m)\frac{ds_2}{s_2}\frac{ds_1}{s_1}.
\end{multline*}
\end{prop}

\section{Solution of the auxiliar equation (\ref{eaux2}) }\label{sec4}

In this section, we consider the assumptions made on the elements involved in the construction of the main problem in Section~\ref{sec3}. After the two reductions of the main problem performed in Section~\ref{sec21} and Section~\ref{sec23} one arrives at auxiliary equation (\ref{eaux2}). In order to solve this auxiliary equation, we first clarify some preliminary results.

The proof of the following result is analogous to that of Lemma 2,~\cite{ma24}, which considers a more general power $\Delta_j$ involved in the definition of $P_{m,j}$, for $j=1,2$. In that result, $\Delta_j=2$.

\begin{prop}\label{prop426}
Let us prescribe the sectorial domains $S_{Q_j,R_j}$ for $j\in\{1,2\}$ given by (\ref{e114}) and endowing the property (\ref{e113}), in a way that the next features hold. For $j\in\{1,2\}$, one can choose a radius $0<\rho<1$ and an unbounded sector $S_{d_j}$ edged at 0 with bisecting direction $d_j\in\R$ for which:
\begin{itemize}
\item A constant $\delta_j>0$ can be found such that
\begin{equation}\label{e499b}
\left|\frac{Q_j(im)}{R_j(im)}-\cosh(\alpha_j(k_j\tau_j^{k_j})^{\Delta_j})\right|>\delta_j,
\end{equation}
for all $\tau_j\in S_{d_j}\cup D(0,\rho)$ and all $m\in\R$.
\item A constant $\nu_j>0$ can be singled out with 
\begin{equation}\label{e500b}
|\cos(k_j\Delta_j\theta_j)|>\nu_j,
\end{equation}
for all $\theta_j\in\R$ such that $e^{i\theta_j}\in S_{d_j}$.
\end{itemize}
For such a radius $\rho>0$ and sectors $S_{d_j}$, $j\in\{1,2\}$, the next lower bounds hold. Given $j\in\{1,2\}$, there exist positive constants $\delta_{S_{d_j},k_j,\alpha_j}$ and $\Delta_{S_{d_j},k_j}$ (only depending on $k_j,\Delta_j$ and $S_{d_j}$) such that 

\begin{equation}\label{e415}
|P_{m,j}(\tau_j)|\ge |R_j(im)|\delta_{S_{d_j},k_j,\alpha_j}\exp\left(\alpha_jk_j^{\Delta_j}\Delta_{S_{d_j},k_j}|\tau_j|^{\Delta_j k_j}\right),
\end{equation}
for every $\tau_j\in S_{d_j}\cup D(0,\rho)$, and all $m\in\R$.
\end{prop}

Let us consider the mapping $\mathcal{H}$ defined by
\begin{multline}
\mathcal{H}(\omega(\boldsymbol{\tau},m)):=\frac{1}{P_{m,1}(\tau_1)P_{m,2}(\tau_2)} \label{eaux3}\\
\times\left[\sum_{\substack{\underline{\ell}=(\ell_0,\ldots,\ell_7)\in\mathcal{A}\\ (\ell_4,\ell_5,\ell_6,\ell_7)=(1,0,0,1)\\ \ell_0=0,\ell_2=0}}\frac{1}{(2\pi)^{1/2}}\int_{-\infty}^{\infty}A_{\underline{\ell}}(m-m_1)(k_1\tau_1^{k_1})^{\ell_1}(k_2\tau_2^{k_2})^{\ell_3}\omega(\tau_1,\tau_2,m_1)R_{\underline{\ell}}(im_1)dm_1\right.\\
+\sum_{\substack{\underline{\ell}=(\ell_0,\ldots,\ell_7)\in\mathcal{A}\\ (\ell_4,\ell_5,\ell_6,\ell_7)=(1,0,0,1)\\ \ell_0\ge 1,\ell_2=0}}\frac{1}{(2\pi)^{1/2}}\int_{-\infty}^{\infty}A_{\underline{\ell}}(m-m_1) C_{k_1,\ell_0,\ell_1}(\omega)(\tau_1,\tau_2,m_1)R_{\underline{\ell}}(im_1)dm_1\\
+\sum_{\substack{\underline{\ell}=(\ell_0,\ldots,\ell_7)\in\mathcal{A}\\ (\ell_4,\ell_5,\ell_6,\ell_7)=(1,0,0,1)\\ \ell_0=0,\ell_2\ge 1}}\frac{1}{(2\pi)^{1/2}}\int_{-\infty}^{\infty}A_{\underline{\ell}}(m-m_1) C_{k_2,\ell_2,\ell_3}(\omega)(\tau_1,\tau_2,m_1)R_{\underline{\ell}}(im_1)dm_1\\
+\sum_{\substack{\underline{\ell}=(\ell_0,\ldots,\ell_7)\in\mathcal{A}\\ (\ell_4,\ell_5,\ell_6,\ell_7)=(1,0,0,1)\\ \ell_0\ge 1,\ell_2\ge 1}}\frac{1}{(2\pi)^{1/2}}\int_{-\infty}^{\infty}A_{\underline{\ell}}(m-m_1) C_{k_1,k_2,\ell_0,\ell_2,\ell_1,\ell_3}(\omega)(\tau_1,\tau_2,m_1)R_{\underline{\ell}}(im_1)dm_1\\
+\sum_{\substack{\underline{\ell}=(\ell_0,\ldots,\ell_7)\in\mathcal{A}\\ \ell_4,\ell_7\ge2, \ell_5,\ell_6\ge1\\ \ell_0=0,\ell_2=0}}\frac{1}{(2\pi)^{1/2}}\int_{-\infty}^{\infty}A_{\underline{\ell}}(m-m_1) 
\int_0^1\int_0^1\int_{\gamma_{\frac{k_1}{\ell_4},\tau_1^{\ell_4}\tau_2^{\ell_5}}}
\int_{\gamma_{\frac{k_2}{\ell_7},\tau_1^{\ell_6}\tau_2^{\ell_7}}}
\int_{\gamma_{\frac{k_2}{\ell_5}}}
\int_{\gamma_{\frac{k_1}{\ell_6}}}\\
(k_1\mathcal{A}_1^{k_1})^{\ell_1}(k_2\mathcal{A}_2^{k_2})^{\ell_3}\omega(\mathcal{A}_1,\mathcal{A}_2,m_1)\mathbb{K}_{\overline{\ell},k_1,k_2}(\boldsymbol{t},\boldsymbol{\xi},\tau_1^{\ell_4}\tau_2^{\ell_5},\tau_1^{\ell_6}\tau_2^{\ell_7})dt_1dt_2d\xi_1d\xi_2d\xi_3d\xi_4 R_{\underline{\ell}}(im_1)dm_1\\
+\sum_{\substack{\underline{\ell}=(\ell_0,\ldots,\ell_7)\in\mathcal{A}\\ \ell_4,\ell_7\ge2, \ell_5,\ell_6\ge1\\ \ell_0\ge1,\ell_2=0}}\frac{1}{(2\pi)^{1/2}}\int_{-\infty}^{\infty}A_{\underline{\ell}}(m-m_1) 
\int_0^1\int_0^1\int_{\gamma_{\frac{k_1}{\ell_4},\tau_1^{\ell_4}\tau_2^{\ell_5}}}
\int_{\gamma_{\frac{k_2}{\ell_7},\tau_1^{\ell_6}\tau_2^{\ell_7}}}
\int_{\gamma_{\frac{k_2}{\ell_5}}}
\int_{\gamma_{\frac{k_1}{\ell_6}}}\\
(k_2\mathcal{A}_2^{k_2})^{\ell_3}C_{k_1,\ell_0,\ell_1}(\omega)(\mathcal{A}_1,\mathcal{A}_2,m_1)\mathbb{K}_{\overline{\ell},k_1,k_2}(\boldsymbol{t},\boldsymbol{\xi},\tau_1^{\ell_4}\tau_2^{\ell_5},\tau_1^{\ell_6}\tau_2^{\ell_7})dt_1dt_2d\xi_1d\xi_2d\xi_3d\xi_4 R_{\underline{\ell}}(im_1)dm_1\\
+\sum_{\substack{\underline{\ell}=(\ell_0,\ldots,\ell_7)\in\mathcal{A}\\ \ell_4,\ell_7\ge2, \ell_5,\ell_6\ge1\\ \ell_0=0,\ell_2\ge1}}\frac{1}{(2\pi)^{1/2}}\int_{-\infty}^{\infty}A_{\underline{\ell}}(m-m_1) 
\int_0^1\int_0^1\int_{\gamma_{\frac{k_1}{\ell_4},\tau_1^{\ell_4}\tau_2^{\ell_5}}}
\int_{\gamma_{\frac{k_2}{\ell_7},\tau_1^{\ell_6}\tau_2^{\ell_7}}}
\int_{\gamma_{\frac{k_2}{\ell_5}}}
\int_{\gamma_{\frac{k_1}{\ell_6}}}\\
(k_1\mathcal{A}_1^{k_1})^{\ell_1}C_{k_2,\ell_2,\ell_3}(\omega)(\mathcal{A}_1,\mathcal{A}_2,m_1)\mathbb{K}_{\overline{\ell},k_1,k_2}(\boldsymbol{t},\boldsymbol{\xi},\tau_1^{\ell_4}\tau_2^{\ell_5},\tau_1^{\ell_6}\tau_2^{\ell_7})dt_1dt_2d\xi_1d\xi_2d\xi_3d\xi_4 R_{\underline{\ell}}(im_1)dm_1\\
+\sum_{\substack{\underline{\ell}=(\ell_0,\ldots,\ell_7)\in\mathcal{A}\\ \ell_4,\ell_7\ge2, \ell_5,\ell_6\ge1\\ \ell_0\ge1,\ell_2\ge1}}\frac{1}{(2\pi)^{1/2}}\int_{-\infty}^{\infty}A_{\underline{\ell}}(m-m_1) 
\int_0^1\int_0^1\int_{\gamma_{\frac{k_1}{\ell_4},\tau_1^{\ell_4}\tau_2^{\ell_5}}}
\int_{\gamma_{\frac{k_2}{\ell_7},\tau_1^{\ell_6}\tau_2^{\ell_7}}}
\int_{\gamma_{\frac{k_2}{\ell_5}}}
\int_{\gamma_{\frac{k_1}{\ell_6}}}\\
C_{k_1,k_2,\ell_0,\ell_2,\ell_1,\ell_3}(\omega)(\mathcal{A}_1,\mathcal{A}_2,m_1)\mathbb{K}_{\overline{\ell},k_1,k_2}(\boldsymbol{t},\boldsymbol{\xi},\tau_1^{\ell_4}\tau_2^{\ell_5},\tau_1^{\ell_6}\tau_2^{\ell_7})dt_1dt_2d\xi_1d\xi_2d\xi_3d\xi_4 R_{\underline{\ell}}(im_1)dm_1
\end{multline}
\begin{multline*}
+\tilde{c}\frac{1}{(2\pi)^{1/2}}\int_{-\infty}^{\infty}\left(\tau_1^{k_1}\tau_2^{k_2}\int_0^{\tau_1^{k_1}}\int_0^{\tau_2^{k_2}} \omega((\tau_1^{k_1}-s_1)^{1/k_1},(\tau_2^{k_2}-s_2)^{1/k_2},m-m_1)\tilde{Q}_1(i(m-m_1))\right.\\
\left.\times\omega(s_1^{1/k_1},s_2^{1/k_2},m_1)\tilde{Q}_2(im_1)\frac{ds_1ds_2}{(\tau_1^{k_1}-s_1)s_1(\tau_2^{k_2}-s_2)s_2}\right)dm_1\\
\left.+\sum_{\underline{j}=(j_1,j_2)\in J}\mathcal{F}_{\underline{j}}(m)\tau_1^{j_1}\tau_2^{j_2}\right],
\end{multline*}

We provide upper bounds for the norm of each of the terms involved in $\mathcal{H}$ in terms of the norm $\left\|\cdot\right\|_{(\boldsymbol{\nu},\beta,\mu,\boldsymbol{k},\rho)}$, associated to the space $F^{\boldsymbol{d}}_{(\boldsymbol{\nu},\beta,\mu,\boldsymbol{k},\rho)}$, where $\boldsymbol{k}=(k_1,k_2)$, $\boldsymbol{d}=(d_1,d_2)$ and $\boldsymbol{\nu}=(\nu_1,\nu_2)$. We refer to Section~\ref{secanexo2} for the definition of the Banach space $(F^{\boldsymbol{d}}_{(\boldsymbol{\nu},\beta,\mu,\boldsymbol{k},\rho)},\left\|\cdot\right\|_{(\boldsymbol{\nu},\beta,\mu,\boldsymbol{k},\rho)})$, and~\cite{ma24} for its main properties.

 Some of them are direct consequences of previous results, so we omit their proof and refer the reader to the corresponding result in the reference. It is worth mentioning that conditions (\ref{e110}) and (\ref{e142}) are needed at this point.

\begin{lemma}[Lemma 3,~\cite{ma24}]\label{lema478}
There exists $C_1>0$, depending on $\mu,R_j,R_{\underline{\ell}},d_j,\boldsymbol{k},\alpha_j,\ell_1,\ell_3$, for $j=1,2$ such that
\begin{multline*}
\left\|\frac{1}{P_{m,1}(\tau_1)P_{m,2}(\tau_2)}\int_{-\infty}^{\infty}A_{\underline{\ell}}(m-m_1)\tau_1^{k_1\ell_1}\tau_2^{k_2\ell_3}\omega(\tau_1,\tau_2,m_1)R_{\underline{\ell}}(im_1)dm_1\right\|_{(\boldsymbol{\nu},\beta,\mu,\boldsymbol{k},\rho)}\\
\le C_1\mathfrak{A}_{\underline{\ell}}\left\|\omega\right\|_{(\boldsymbol{\nu},\beta,\mu,\boldsymbol{k},\rho)}
\end{multline*}
for every $\omega\in F^{\boldsymbol{d}}_{(\boldsymbol{\nu},\beta,\mu,\boldsymbol{k},\rho)}$.
\end{lemma}

The following results hold.
\begin{lemma}[Lemma 4,~\cite{ma24}]\label{lema479}
Assume $\ell_0\ge 1$. There exists a constant $C_{2,1}>0$, depending on $\mu,R_j,R_{\underline{\ell}},d_j,\boldsymbol{k},\alpha_j,\ell_0,\ell_1$, for $j=1,2$ such that
\begin{multline*}
\left\|\frac{1}{P_{m,1}(\tau_1)P_{m,2}(\tau_2)}\int_{-\infty}^{\infty}A_{\underline{\ell}}(m-m_1) C_{k_1,\ell_0,\ell_1}(\omega)(\tau_1,\tau_2,m_1)R_{\underline{\ell}}(im_1)dm_1\right\|_{(\boldsymbol{\nu},\beta,\mu,\boldsymbol{k},\rho)}\\
\le C_{2,1}\mathfrak{A}_{\underline{\ell}}\left\|\omega\right\|_{(\boldsymbol{\nu},\beta,\mu,\boldsymbol{k},\rho)}
\end{multline*}
for every $\omega\in F^{\boldsymbol{d}}_{(\boldsymbol{\nu},\beta,\mu,\boldsymbol{k},\rho)}$.
\end{lemma}

\begin{lemma}[Lemma 4,~\cite{ma24}]\label{lema480}
Assume $\ell_2\ge 1$. There exists a constant $C_{2,2}>0$, depending on $\mu,R_j,R_{\underline{\ell}},d_j,\boldsymbol{k},\alpha_j,\ell_2,\ell_3$, for $j=1,2$ such that
\begin{multline*}
\left\|\frac{1}{P_{m,1}(\tau_1)P_{m,2}(\tau_2)}\int_{-\infty}^{\infty}A_{\underline{\ell}}(m-m_1) C_{k_2,\ell_2,\ell_3}(\omega)(\tau_1,\tau_2,m_1)R_{\underline{\ell}}(im_1)dm_1\right\|_{(\boldsymbol{\nu},\beta,\mu,\boldsymbol{k},\rho)}\\
\le C_{2,2}\mathfrak{A}_{\underline{\ell}}\left\|\omega\right\|_{(\boldsymbol{\nu},\beta,\mu,\boldsymbol{k},\rho)}
\end{multline*}
for every $\omega\in F^{\boldsymbol{d}}_{(\boldsymbol{\nu},\beta,\mu,\boldsymbol{k},\rho)}$.
\end{lemma}

A straightforward adaptation of the proof of Lemma 4,~\cite{ma24} to the several variable framework is valid to achieve that of the next result. Upper estimates can be obtained by distinguishing the two variables $\tau_1$ and $\tau_2$.

\begin{lemma}[Lemma 4,~\cite{ma24}]\label{lema481}
Assume $\ell_0,\ell_2\ge 1$. There exists a constant $C_{2}>0$, depending on $\mu,R_j,R_{\underline{\ell}},d_j,\boldsymbol{k},\alpha_j$, $\ell_0,\ell_1,\ell_2,\ell_3$, for $j=1,2$ such that
\begin{multline*}
\left\|\frac{1}{P_{m,1}(\tau_1)P_{m,2}(\tau_2)}\int_{-\infty}^{\infty}A_{\underline{\ell}}(m-m_1) C_{k_1,k_2,\ell_0,\ell_2,\ell_1,\ell_3}(\omega)(\tau_1,\tau_2,m_1)R_{\underline{\ell}}(im_1)dm_1\right\|_{(\boldsymbol{\nu},\beta,\mu,\boldsymbol{k},\rho)}\\
\le C_{2}\mathfrak{A}_{\underline{\ell}}\left\|\omega\right\|_{(\boldsymbol{\nu},\beta,\mu,\boldsymbol{k},\rho)}
\end{multline*}
for every $\omega\in F^{\boldsymbol{d}}_{(\boldsymbol{\nu},\beta,\mu,\boldsymbol{k},\rho)}$.
\end{lemma}


The next results deal with the upper bounds associated to the norm of the terms in which Mahler operators appear in $\mathcal{H}$.

\begin{lemma}\label{lema521}
Let $\underline{\ell}=(\ell_0,\ldots,\ell_7)\in\mathcal{A}$ with $\ell_4,\ell_7\ge2$ and $\ell_5,\ell_6\ge1$ with $\ell_0=\ell_2=0$. Then, there exists $C_3>0$ such that 

\begin{multline*}
\left\|\int_{-\infty}^{\infty}\frac{1}{P_{m,1}(\tau_1)P_{m,2}(\tau_2)}A_{\underline{\ell}}(m-m_1) 
\int_0^1\int_0^1\int_{\gamma_{\frac{k_1}{\ell_4},\tau_1^{\ell_4}\tau_2^{\ell_5}}}
\int_{\gamma_{\frac{k_2}{\ell_7},\tau_1^{\ell_6}\tau_2^{\ell_7}}}
\int_{\gamma_{\frac{k_2}{\ell_5}}}
\int_{\gamma_{\frac{k_1}{\ell_6}}}\right.\\
\left.\times \mathcal{A}_1^{k_1\ell_1}\mathcal{A}_2^{k_2\ell_3}\omega(\mathcal{A}_1,\mathcal{A}_2,m_1)\mathbb{K}_{\overline{\ell},k_1,k_2}(\boldsymbol{t},\boldsymbol{\xi},\tau_1^{\ell_4}\tau_2^{\ell_5},\tau_1^{\ell_6}\tau_2^{\ell_7})dt_1dt_2d\xi_1d\xi_2d\xi_3d\xi_4 R_{\underline{\ell}}(im_1)dm_1\right\|_{(\boldsymbol{\nu},\beta,\mu,\boldsymbol{k},\rho)}\\
\le C_3  \mathfrak{A}_{\underline{\ell}}\left\|\omega\right\|_{(\boldsymbol{\nu},\beta,\mu,\boldsymbol{k},\rho)},
\end{multline*}
valid for every $\omega\in F^{\boldsymbol{d}}_{(\boldsymbol{\nu},\beta,\mu,\boldsymbol{k},\rho)}$.
\end{lemma}
\begin{proof}
Let $\omega\in F^{\boldsymbol{d}}_{(\boldsymbol{\nu},\beta,\mu,\boldsymbol{k},\rho)}$. For every $\tau_j\in S_{d_j}\cup D(0,\rho)$ and $m\in\R$ we plan to provide upper bounds for
\begin{multline}
\left|\int_{-\infty}^{\infty}\frac{1}{P_{m,1}(\tau_1)P_{m,2}(\tau_2)}A_{\underline{\ell}}(m-m_1) 
\int_0^1\int_0^1\int_{\gamma_{\frac{k_1}{\ell_4},\tau_1^{\ell_4}\tau_2^{\ell_5}}}
\int_{\gamma_{\frac{k_2}{\ell_7},\tau_1^{\ell_6}\tau_2^{\ell_7}}}
\int_{\gamma_{\frac{k_2}{\ell_5}}}
\int_{\gamma_{\frac{k_1}{\ell_6}}}\right.\\
\left.\mathcal{A}_1^{k_1\ell_1}\mathcal{A}_2^{k_2\ell_3}\omega(\mathcal{A}_1,\mathcal{A}_2,m_1) \mathbb{K}_{\overline{\ell},k_1,k_2}(\boldsymbol{t},\boldsymbol{\xi},\tau_1^{\ell_4}\tau_2^{\ell_5},\tau_1^{\ell_6}\tau_2^{\ell_7})dt_1dt_2d\xi_1d\xi_2d\xi_3d\xi_4 R_{\underline{\ell}}(im_1)dm_1\right|\\
\times(1+|m|)^{\mu}e^{\beta|m|}\frac{1+|\tau_1|^{2k_1}}{|\tau_1|}\frac{1+|\tau_2|^{2k_2}}{|\tau_2|}\exp\left(-\nu_1|\tau_1|^{k_1}-\nu_2|\tau_2|^{k_2}\right)\label{e573}
\end{multline}

We first observe that $|\mathcal{A}_j|\le \rho$ for $j=1,2$. Indeed, all the Hankel paths $\gamma_{\frac{k_1}{\ell_4},\tau_1^{\ell_4}\tau_2^{\ell_5}}$, $\gamma_{\frac{k_2}{\ell_7},\tau_1^{\ell_6}\tau_2^{\ell_7}}$, $\gamma_{\frac{k_2}{\ell_5}}$, $\gamma_{\frac{k_1}{\ell_6}}$ are all contained in the disc $D(0,\rho/2)$ by the construction given in Proposition~\ref{prop189}. Since we assume that $0<\rho<1$, we deduce that $\xi_1,\xi_3\in D(0,\rho)$ for $\xi_1\in \gamma_{\frac{k_1}{\ell_4},\tau_1^{\ell_4}\tau_2^{\ell_5}}$ and $\xi_3\in \gamma_{\frac{k_2}{\ell_5}}$ together with $\xi_2,\xi_4\in D(0,\rho)$ for $\xi_2\in \gamma_{\frac{k_2}{\ell_7},\tau_1^{\ell_6}\tau_2^{\ell_7}}$ and $\xi_4\in \gamma_{\frac{k_1}{\ell_6}}$.

The application of Lemma~\ref{lema7} allows to upper estimate 
$$|\omega(\mathcal{A}_1,\mathcal{A}_2,m_1)|\le M_{\rho,\boldsymbol{\nu},\boldsymbol{k}} \frac{1}{(1+|m_1|)^{\mu}}e^{-\beta|m_1|}\left\|\omega\right\|_{(\boldsymbol{\nu},\beta,\mu,\boldsymbol{k},\rho)}.$$


We now focus on the kernel function in (\ref{e275}). At first, we can find two constants $M_1,M_2>0$ depending on $\boldsymbol{k},\underline{\ell}$, such that
\begin{equation}\label{e582}
\int_{|\gamma_{\frac{k_2}{\ell_5}}|}\left|\xi_3^{k_1\ell_1}\exp\left(\left(\frac{1}{\xi_3}\right)^{\frac{k_2}{\ell_5}}\right)\frac{1}{\xi_3^{\frac{k_2}{\ell_5}+1}}\right|d|\xi_3|\le M_1,\hbox{ and }
\end{equation}
\begin{equation}\label{e583}
\int_{|\gamma_{\frac{k_1}{\ell_6}}|}\left|\xi_4^{k_2\ell_3}\exp\left(\left(\frac{1}{\xi_4}\right)^{\frac{k_1}{\ell_6}}\right)\frac{1}{\xi_4^{\frac{k_1}{\ell_6}+1}}\right|d|\xi_4|\le M_2.
\end{equation}


From the assumptions (i) and (ii) made on the elements in $\mathcal{A}$, one has that
$$M_3=\int_0^1(1-t_1)^{\frac{\ell_4}{k_1}k_1\ell_1-1}t_1^{\frac{\ell_6}{k_1}k_2\ell_3-1}dt_1<\infty,$$
and
$$M_4=\int_0^1(1-t_2)^{\frac{\ell_5}{k_2}k_1\ell_1-1}t_2^{\frac{\ell_7}{k_2}k_2\ell_3-1}dt_2<\infty.$$
We observe the existence of $\mathfrak{R}_j,\mathfrak{R}_{\underline{\ell}}>0$ such that $|R_{j}(im)|\ge\mathfrak{R}_j(1+|m|)^{\hbox{deg}(R_j)}$ for $m\in\R$ and $j=1,2$ and $|R_{\underline{\ell}}(im)|\le\mathfrak{R}_{\underline{\ell}}(1+|m|)^{\hbox{deg}(R_{\underline{\ell}})}$, for all $\underline{\ell}\in\mathcal{A}$ and every $m\in\R$, and 
$$|A_{\underline{\ell}}(m-m_1)|\le \mathfrak{A}_{\underline{\ell}}(1
+|m-m_1|)^{-\mu}e^{-\beta|m-m_1|}$$
 for all $m,m_1\in\R$. These facts, together with (\ref{e110}) allow us to apply Lemma 2.2 in~\cite{cota2} which guarantees the existence of $M_5>0$ such that
$$\sup_{m\in\R}(1+|m|)^{\mu-\hbox{deg}(R_1)-\hbox{deg}(R_2)}\int_{-\infty}^{\infty}\frac{1}{(1+|m-m_1|)^{\mu}(1+|m_1|)^{\mu-\hbox{deg}(R_{\underline{\ell}})}}dm_1\le M_5.$$  

The previous estimates, together with (\ref{e415}) and the fact that 
$$|m|\le|m-m_1|+|m_1|,\quad m,m_1\in\R,$$
guarantee that (\ref{e573}) is upper estimated by

\begin{multline}
M_6\mathfrak{A}_{\underline{\ell}}
\left\|\omega\right\|_{(\boldsymbol{\nu},\beta,\mu,\boldsymbol{k},\rho)}\frac{1}{\exp\left(\alpha_1k_1^{\Delta_1}\Delta_{S_{d_1},k_1}|\tau_1|^{\Delta_1 k_1}+\alpha_2k_2^{\Delta_2}\Delta_{S_{d_2},k_2}|\tau_2|^{\Delta_2 k_2}\right)}\\
\times\left[\int_{|\gamma_{\frac{k_1}{\ell_4},\tau_1^{\ell_4}\tau_2^{\ell_5}}|}|\xi_1|^{k_1\ell_1}|\mathbb{D}_{k_1,\frac{k_1}{\ell_4}}(\xi_1,\tau_1^{\ell_4}\tau_2^{\ell_5})| d|\xi_1|\right]
\times\left[\int_{|\gamma_{\frac{k_2}{\ell_7},\tau_1^{\ell_6}\tau_2^{\ell_7}}|}|\xi_2|^{k_2\ell_3}
|\mathbb{D}_{k_2,\frac{k_2}{\ell_7}}(\xi_2,\tau_1^{\ell_6}\tau_2^{\ell_7})| d|\xi_2|\right]\\
\times \frac{1+|\tau_1|^{2k_1}}{|\tau_1|}\frac{1+|\tau_2|^{2k_2}}{|\tau_2|}\exp\left(-\nu_1|\tau_1|^{k_1}-\nu_2|\tau_2|^{k_2}\right)\label{e612}
\end{multline}
with $M_6=\frac{k_1^2/\ell_4}{2\pi}\frac{k_2^2/\ell_7}{2\pi}\frac{k_2/\ell_5}{2\pi}\frac{k_1/\ell_6}{2\pi}\frac{M_{\rho,\boldsymbol{\nu},\boldsymbol{k}}M_1M_2M_3M_4M_5\mathfrak{R}_{\underline{\ell}}}{\mathfrak{R}_{1}\mathfrak{R}_{2}\delta_{S_{d_1},k_1,\alpha_1}\delta_{S_{d_2},k_2,\alpha_2}}$.

At this point, we recall Proposition 7~\cite{ma24} before stating upper bounds for (\ref{e612}). 

\begin{prop}[Proposition 7,~\cite{ma24}]\label{prop588}
Let $1\le k<k'$ be rational numbers. 

1) There exists a constant $M_{k,k'}>0$ (depending on $k,k'$) such that
$$|\mathbb{D}_{k',k}(\xi,h)|\le \frac{M_{k',k}}{k|\xi|^{k+1}}|h|^{k}\left|\frac{h}{\xi}\right|^{\frac{k^2}{k'-k}}\exp\left(\left|\frac{h}{\xi}\right|^{\frac{kk'}{k'-k}}\right),$$
for every $h\in\C^{\star}$, and all $\xi\in\gamma_{k,h}$ provided that $|\xi|\le |h|$. Here, $\gamma_{k,h}$ is an adequate Hankel path. 

2) There exists a constant $M_{k',k,1,2}>0$ (depending on $k,k'$) such that 
$$|\mathbb{D}_{k',k}(\xi,h)|\le M_{k',k,1,2}\frac{|h|^{k}}{k|\xi|^{k+1}},$$
valid for all $h\in\C^{\star}$, and $\xi\in\gamma_{k,h,1}$ or $\xi\in\gamma_{k,h,2}$, i.e. the subpath of $\gamma_{k,h}$ consisting of oriented segments with the origin being one of its endpoints. 
\end{prop}

Let us first study 
\begin{equation}\label{e610}
J_1:=\int_{|\gamma_{\frac{k_1}{\ell_4},\tau_1^{\ell_4}\tau_2^{\ell_5}}|}|\xi_1|^{k_1\ell_1}|\mathbb{D}_{k_1,\frac{k_1}{\ell_4}}(\xi_1,\tau_1^{\ell_4}\tau_2^{\ell_5})| d|\xi_1|.
\end{equation}
Let $\tau_j\in S_{d_j}\cup D(0,\rho)$ for $j=1,2$. We split the integration path $\gamma_{\frac{k_1}{\ell_4},\tau_1^{\ell_4}\tau_2^{\ell_5}}$ in different parts:
\begin{itemize}
\item[(i)] Assume $\xi_1\in \gamma_{\frac{k_1}{\ell_4},\tau_1^{\ell_4}\tau_2^{\ell_5},1}$. Then, 2) in Proposition~\ref{prop588} and (ii) in the hypotheses of $\mathcal{A}$ of Section~\ref{sec3} guarantee that
\begin{multline*}
\int_{|\gamma_{\frac{k_1}{\ell_4},\tau_1^{\ell_4}\tau_2^{\ell_5},1}|}|\xi_1|^{k_1\ell_1}|\mathbb{D}_{k_1,\frac{k_1}{\ell_4}}(\xi_1,\tau_1^{\ell_4}\tau_2^{\ell_5})| d|\xi_1|\le \int_{0}^{\rho/2}|\xi_1|^{k_1\ell_1}M_{k_1,\frac{k_1}{\ell_4},1,2}\frac{(|\tau_1|^{\ell_4}|\tau_2|^{\ell_5})^{\frac{k_1}{\ell_4}}}{k_1/\ell_4|\xi_1|^{\frac{k_1}{\ell_4}+1}}d|\xi_1|\\
\le M_7(|\tau_1|^{\ell_4}|\tau_2|^{\ell_5})^{\frac{k_1}{\ell_4}},
\end{multline*} 
for some $M_7>0$.
\item[(ii)] Analogous estimates guarantee that
\begin{equation}\label{e611}
\int_{|\gamma_{\frac{k_1}{\ell_4},\tau_1^{\ell_4}\tau_2^{\ell_5},2}|}|\xi_1|^{k_1\ell_1}|\mathbb{D}_{k_1,\frac{k_1}{\ell_4}}(\xi_1,\tau_1^{\ell_4}\tau_2^{\ell_5})| d|\xi_1|\le M_8(|\tau_1|^{\ell_4}|\tau_2|^{\ell_5})^{\frac{k_1}{\ell_4}},
\end{equation}
for some $M_8>0$.
\item[(iii)] We split the path $\gamma_{\frac{k_1}{\ell_4},\tau_1^{\ell_4}\tau_2^{\ell_5},3}$ in two parts: $\gamma_{\frac{k_1}{\ell_4},\tau_1^{\ell_4}\tau_2^{\ell_5},3,1}$ stands for the elements with $|\xi_1|\le |\tau_1^{\ell_4}\tau_2^{\ell_5}|$ and $\gamma_{\frac{k_1}{\ell_4},\tau_1^{\ell_4}\tau_2^{\ell_5},3,2}$ stands for the elements with $|\xi_1|> |\tau_1^{\ell_4}\tau_2^{\ell_5}|$. In view of (ii) in the hypotheses of $\mathcal{A}$ of Section~\ref{sec3}, one has that
\begin{multline}\label{e616}
\int_{|\gamma_{\frac{k_1}{\ell_4},\tau_1^{\ell_4}\tau_2^{\ell_5},3}|}|\xi_1|^{k_1\ell_1}|\mathbb{D}_{k_1,\frac{k_1}{\ell_4}}(\xi_1,\tau_1^{\ell_4}\tau_2^{\ell_5})| d|\xi_1|\\
=\int_{|\gamma_{\frac{k_1}{\ell_4},\tau_1^{\ell_4}\tau_2^{\ell_5},3,1}|}|\xi_1|^{k_1\ell_1}|\mathbb{D}_{k_1,\frac{k_1}{\ell_4}}(\xi_1,\tau_1^{\ell_4}\tau_2^{\ell_5})| d|\xi_1|\\
+\int_{|\gamma_{\frac{k_1}{\ell_4},\tau_1^{\ell_4}\tau_2^{\ell_5},3,2}|}|\xi_1|^{k_1\ell_1}|\mathbb{D}_{k_1,\frac{k_1}{\ell_4}}(\xi_1,\tau_1^{\ell_4}\tau_2^{\ell_5})| d|\xi_1|=:J_{1,1}+J_{1,2}.
\end{multline}
In view of 1) in Proposition~\ref{prop588} one derives
$$J_{1,1}\le 2\pi\left(\frac{\rho}{2}\right)^{k_1\ell_1}\frac{M_{k_1,\frac{k_1}{\ell_4}}}{k_1/\ell_4(\rho/2)^{\frac{k_1}{\ell_4}+1}}|\tau_1^{\ell_4}\tau_2^{\ell_5}|^{\frac{k_1}{\ell_4}}\left|\frac{\tau_1^{\ell_4}\tau_2^{\ell_5}}{\rho/2}\right|^{\frac{k_1}{\ell_4(\ell_4-1)}}\exp\left(\frac{|\tau_1^{\ell_4}\tau_2^{\ell_5}|^{\frac{k_1}{\ell_4-1}}}{(\rho/2)^{\frac{k_1}{\ell_4-1}}}\right).$$
On the other hand, if we assume that $|\tau_1^{\ell_4}\tau_2^{\ell_5}|<|\xi_1|=\rho/2$, we observe from the proof of Proposition 7,~\cite{ma24} that
$$\mathbb{D}_{k',k}(\xi,h)=\frac{1}{k\xi^{k+1}}h^k\boldsymbol{D}_{k'/k}((\frac{h}{\xi})^k),$$
for some entire function $z\mapsto \boldsymbol{D}_{k'/k}(z)$. This entails that
$$|\mathbb{D}_{k_1,k_1/\ell_4}(\xi_1,\tau_1^{\ell_4}\tau_2^{\ell_5})|\le \frac{1}{k_1/\ell_4(\rho/2)^{\frac{k_1}{\ell_4}+1}}|\tau_1^{\ell_4}\tau_2^{\ell_5}|^{\frac{k_1}{\ell_4}}\left(\max_{z\in \overline{D}(0,1)}|\boldsymbol{D}_{\ell_4}(z)|\right),$$
from which we derive
$$J_{1,2}\le 2\pi (\rho/2)^{k_1\ell_1}\frac{1}{k_1/\ell_4(\rho/2)^{\frac{k_1}{\ell_4}+1}}|\tau_1^{\ell_4}\tau_2^{\ell_5}|^{\frac{k_1}{\ell_4}}\left(\max_{z\in \overline{D}(0,1)}|\boldsymbol{D}_{\ell_4}(z)|\right).$$
\end{itemize}
We conclude the existence of $M_9,M_{10}>0$ such that
\begin{equation}\label{e629}
J_{1}\le M_9\max\left\{|\tau_1^{\ell_4}\tau_2^{\ell_5}|^{\frac{k_1}{\ell_4}},|\tau_1^{\ell_4}\tau_2^{\ell_5}|^{\frac{k_1}{\ell_4-1}}\right\}\exp\left(M_{10}|\tau_1^{\ell_4}\tau_2^{\ell_5}|^{\frac{k_1}{\ell_4-1}}\right).
\end{equation}

In an analogous way, we derive the existence of $M_{11},M_{12}>0$ such that
\begin{multline}
J_2:=\int_{|\gamma_{\frac{k_2}{\ell_7},\tau_1^{\ell_6}\tau_2^{\ell_7}}|}|\xi_2|^{k_2\ell_3}|\mathbb{D}_{k_2,\frac{k_2}{\ell_7}}(\xi_2,\tau_1^{\ell_6}\tau_2^{\ell_7})| d|\xi_2|\\
\le M_{11}\max\left\{|\tau_1^{\ell_6}\tau_2^{\ell_7}|^{\frac{k_2}{\ell_7}},|\tau_1^{\ell_6}\tau_2^{\ell_7}|^{\frac{k_2}{\ell_7-1}}\right\}\exp\left(M_{12}|\tau_1^{\ell_6}\tau_2^{\ell_7}|^{\frac{k_2}{\ell_7-1}}\right).\label{e630}
\end{multline}

Therefore, the expression in (\ref{e612}) is upper estimated by
\begin{multline*}
M_6M_{9}M_{11}\mathfrak{A}_{\underline{\ell}}\left\|\omega\right\|_{(\boldsymbol{\nu},\beta,\mu,\boldsymbol{k},\rho)}\frac{1}{\exp\left(\alpha_1k_1^{\Delta_1}\Delta_{S_{d_1},k_1}|\tau_1|^{\Delta_1 k_1}+\alpha_2k_2^{\Delta_2}\Delta_{S_{d_2},k_2}|\tau_2|^{\Delta_2 k_2}\right)}\\
\times \frac{1+|\tau_1|^{2k_1}}{|\tau_1|}\frac{1+|\tau_2|^{2k_2}}{|\tau_2|}\exp\left(-\nu_1|\tau_1|^{k_1}-\nu_2|\tau_2|^{k_2}\right)\exp\left(M_{10}|\tau_1^{\ell_4}\tau_2^{\ell_5}|^{\frac{k_1}{\ell_4-1}}\right)\exp\left(M_{12}|\tau_1^{\ell_6}\tau_2^{\ell_7}|^{\frac{k_2}{\ell_7-1}}\right)\\
\times\max\left\{|\tau_1^{\ell_4}\tau_2^{\ell_5}|^{\frac{k_1}{\ell_4}},|\tau_1^{\ell_4}\tau_2^{\ell_5}|^{\frac{k_1}{\ell_4-1}}\right\}\max\left\{|\tau_1^{\ell_6}\tau_2^{\ell_7}|^{\frac{k_2}{\ell_7}},|\tau_1^{\ell_6}\tau_2^{\ell_7}|^{\frac{k_2}{\ell_7-1}}\right\}.
\end{multline*}

Observe that 
$$\frac{1}{|\tau_1\tau_2|}\max\left\{|\tau_1^{\ell_4}\tau_2^{\ell_5}|^{\frac{k_1}{\ell_4}},|\tau_1^{\ell_4}\tau_2^{\ell_5}|^{\frac{k_1}{\ell_4-1}}\right\}\max\left\{|\tau_1^{\ell_6}\tau_2^{\ell_7}|^{\frac{k_2}{\ell_7}},|\tau_1^{\ell_6}\tau_2^{\ell_7}|^{\frac{k_2}{\ell_7-1}}\right\}$$
is bounded for $(\tau_1,\tau_2)\in D(0,\rho)^2$ under the conditions set in Section~\ref{sec3}.

We recall Young's inequality which states that for all $a,b\ge0$ and any $p,q>1$ such that $1/p+1/q=1$, then $ab\le \frac{a^p}{p}+\frac{b^q}{q}.$

Therefore, for all $p_j,q_j>1$ such that $1/p_j+1/q_j=1$, $j=1,2$, one has
$$\exp\left(M_{10}|\tau_1^{\ell_4}\tau_2^{\ell_5}|^{\frac{k_1}{\ell_4-1}}\right)\le\exp\left(\frac{M_{10}}{p_1}|\tau_1|^{\frac{\ell_4k_1p_1}{\ell_4-1}}\right)\exp\left(\frac{M_{10}}{q_1}|\tau_2|^{\frac{\ell_5k_1q_1}{\ell_4-1}}\right),$$
$$\exp\left(M_{12}|\tau_1^{\ell_6}\tau_2^{\ell_7}|^{\frac{k_2}{\ell_7-1}}\right)\le \exp\left(\frac{M_{12}}{p_2}|\tau_1|^{\frac{\ell_6k_2p_2}{\ell_7-1}}\right)\exp\left(\frac{M_{12}}{q_2}|\tau_2|^{\frac{\ell_7k_2q_2}{\ell_7-1}}\right).$$

In view of the assumption made in (\ref{e115}), the result follows.
\end{proof}

We only sketch the proof of the next result at the steps where it differs from that of Lemma~\ref{lema521}.

\begin{lemma}\label{lema660}
Let $\underline{\ell}=(\ell_0,\ldots,\ell_7)\in\mathcal{A}$ with $\ell_4,\ell_7\ge 2$ and $\ell_5,\ell_6\ge1$ with $\ell_0\ge1$ and $\ell_2=0$. Then, there exists $C_4>0$ such that 
\begin{multline*}
\left\|\int_{-\infty}^{\infty}\frac{1}{P_{m,1}(\tau_1)P_{m,2}(\tau_2)}A_{\underline{\ell}}(m-m_1) 
\int_0^1\int_0^1\int_{\gamma_{\frac{k_1}{\ell_4},\tau_1^{\ell_4}\tau_2^{\ell_5}}}
\int_{\gamma_{\frac{k_2}{\ell_7},\tau_1^{\ell_6}\tau_2^{\ell_7}}}
\int_{\gamma_{\frac{k_2}{\ell_5}}}
\int_{\gamma_{\frac{k_1}{\ell_6}}}\right.\\
\left.\times \mathcal{A}_2^{k_2\ell_3}C_{k_1,\ell_0,\ell_1}(\omega)(\mathcal{A}_1,\mathcal{A}_2,m_1)\right.\\
\left.\times \mathbb{K}_{\overline{\ell},k_1,k_2}(\boldsymbol{t},\boldsymbol{\xi},\tau_1^{\ell_4}\tau_2^{\ell_5},\tau_1^{\ell_6}\tau_2^{\ell_7})dt_1dt_2d\xi_1d\xi_2d\xi_3d\xi_4 R_{\underline{\ell}}(im_1)dm_1\right\|_{(\boldsymbol{\nu},\beta,\mu,\boldsymbol{k},\rho)}\\
\le C_4 \mathfrak{A}_{\underline{\ell}}\left\|\omega\right\|_{(\boldsymbol{\nu},\beta,\mu,\boldsymbol{k},\rho)},
\end{multline*}
valid for every $\omega\in F^{\boldsymbol{d}}_{(\boldsymbol{\nu},\beta,\mu,\boldsymbol{k},\rho)}$.
\end{lemma}
\begin{proof}
Let $\omega\in F^{\boldsymbol{d}}_{(\boldsymbol{\nu},\beta,\mu,\boldsymbol{k},\rho)}$. For all $\tau_j\in S_{d_j}\cup D(0,\rho)$ and $m\in\R$, we consider
\begin{multline}
\left|\int_{-\infty}^{\infty}\frac{1}{P_{m,1}(\tau_1)P_{m,2}(\tau_2)}A_{\underline{\ell}}(m-m_1) 
\int_0^1\int_0^1\int_{\gamma_{\frac{k_1}{\ell_4},\tau_1^{\ell_4}\tau_2^{\ell_5}}}
\int_{\gamma_{\frac{k_2}{\ell_7},\tau_1^{\ell_6}\tau_2^{\ell_7}}}
\int_{\gamma_{\frac{k_2}{\ell_5}}}
\int_{\gamma_{\frac{k_1}{\ell_6}}}\right.\\
\left. \mathcal{A}_2^{k_2\ell_3}C_{k_1,\ell_0,\ell_1}(\omega)(\mathcal{A}_1,\mathcal{A}_2,m_1)\mathbb{K}_{\overline{\ell},k_1,k_2}(\boldsymbol{t},\boldsymbol{\xi},\tau_1^{\ell_4}\tau_2^{\ell_5},\tau_1^{\ell_6}\tau_2^{\ell_7})dt_1dt_2d\xi_1d\xi_2d\xi_3d\xi_4 R_{\underline{\ell}}(im_1)dm_1\right|\\
\times (1+|m|)^{\mu}e^{\beta|m|}\frac{1+|\tau_1|^{2k_1}}{|\tau_1|}\frac{1+|\tau_2|^{2k_2}}{|\tau_2|}\exp\left(-\nu_1|\tau_1|^{k_1}-\nu_2|\tau_2|^{k_2}\right).\label{e683}
\end{multline}

The procedure followed in the proof of Lemma~\ref{lema521} can be mimicked at some points. In particular, the estimates in (\ref{e582}) with $k_1\ell_1$ substituted by $k_1\ell_1+\ell_0$, (\ref{e583}), (\ref{e415}), (\ref{e629}), (\ref{e630}) are valid. We also determine upper estimates for $|C_{k_1,\ell_0,\ell_1}(\omega)(\mathcal{A}_1,\mathcal{A}_2,m_1)|$ for every $m_1\in\R$. Indeed, one has
\begin{multline*}
|C_{k_1,\ell_0,\ell_1}(\omega)(\mathcal{A}_1,\mathcal{A}_2,m_1)|\le \frac{|\mathcal{A}_1|^{k_1}}{\Gamma\left(\frac{\ell_0}{k_1}\right)}\int_0^{|\mathcal{A}_1|^{k_1}}(|\mathcal{A}_1|^{k_1}-h_1)^{\frac{\ell_0}{k_1}-1}k_1^{\ell_1}h_1^{\ell_1}|\omega(h_1^{\frac{1}{k_1}}e^{i\hbox{arg}(\mathcal{A}_1)},\mathcal{A}_2,m_1)|\frac{dh_1}{h_1}\\
\le\frac{1}{\Gamma\left(\frac{\ell_0}{k_1}\right)}\frac{M_{\rho,\boldsymbol{\nu},\boldsymbol{k}}}{(1+|m_1|)^{\mu}}e^{-\beta|m_1|}\left\|\omega\right\|_{(\boldsymbol{\nu},\beta,\mu,\boldsymbol{k},\rho)}|\mathcal{A}_1|^{\ell_0+k_1\ell_1}B(\frac{\ell_0}{k_1},\ell_1),
\end{multline*}
where we have applied the change of variable $h_1=|\mathcal{A}_1|^{k_1}u_1$ in the last integral, and where $B(\cdot,\cdot)$ denotes Beta function.

From the assumption (iv) made on the elements in $\mathcal{A}$, one has
$$\tilde{M}_3=\int_0^1(1-t_1)^{\frac{\ell_4}{k_1}(\ell_0+k_1\ell_1)-1}t_1^{\frac{\ell_6}{k_1}k_2\ell_3-1}dt_1<\infty,$$
and
$$\tilde{M}_4=\int_0^1(1-t_2)^{\frac{\ell_5}{k_2}(\ell_0+k_1\ell_1)-1}t_2^{\frac{\ell_7}{k_2}k_2\ell_3-1}dt_2<\infty.$$

Indeed, the expression (\ref{e683}) can be upper estimated by
\begin{multline}
\frac{\tilde{M}_1M_2\tilde{M}_3\tilde{M}_4M_5k_1^{\ell_1}M_{\rho,\boldsymbol{\nu},\boldsymbol{k}}B(\frac{\ell_0}{k_1},\ell_1)\mathcal{R}_{\underline{\ell}}k_1^2k_2^2}{\mathfrak{R}_1\mathfrak{R}_2\delta_{S_{d_1},k_1,\alpha_1}\delta_{S_{d_2},k_2,\alpha_2}\Gamma(\ell_0/k_1)4\pi^2\ell_4\ell_7}\frac{1}{\exp\left(\alpha_1k_1^{\Delta_1}\Delta_{S_{d_1},k_1}|\tau_1|^{\Delta_1 k_1}+\alpha_2k_2^{\Delta_2}\Delta_{S_{d_2},k_2}|\tau_2|^{\Delta_2 k_2}\right)}\\
\times \frac{k_2/\ell_5}{2\pi}\frac{k_1/\ell_6}{2\pi}\mathfrak{A}_{\underline{\ell}}\left\|\omega\right\|_{(\boldsymbol{\nu},\beta,\mu,\boldsymbol{k},\rho)}
\exp(-\nu_1|\tau_1|^{k_1}-\nu_2|\tau_2|^{k_2})\frac{1+|\tau_1|^{2k_1}}{|\tau_1|}\frac{1+|\tau_2|^{2k_2}}{|\tau_2|}\tilde{J}_1J_2,\label{e702}
\end{multline}
where $\tilde{J}_1$ is obtained after substituting $|\xi_1|^{k_1\ell_1}$ by $|\xi_1|^{k_1\ell_1+\ell_0}$ in the expression of $J_1$ (see (\ref{e610})). Indeed, from the definition of the Hankel path defining $J_1$ one derives $\tilde{J}_1\le (\rho/2)^{\ell_0}J_1$. Therefore, analogous bounds to (\ref{e629}) and (\ref{e630}) are obtained. The expression in (\ref{e702}) is upper estimated as in the last step of the proof of Lemma~\ref{lema521} via the application of Young's inequality. The result follows from here.
\end{proof}

\begin{lemma}\label{lema661}
Let $\underline{\ell}=(\ell_0,\ldots,\ell_7)\in\mathcal{A}$ with $\ell_4,\ell_7\ge 2$ and $\ell_5,\ell_6\ge1$ with $\ell_0=0$ and $\ell_2\ge1$. Then, there exists $C_5>0$ such that 
\begin{multline*}
\left\|\int_{-\infty}^{\infty}\frac{1}{P_{m,1}(\tau_1)P_{m,2}(\tau_2)}A_{\underline{\ell}}(m-m_1) 
\int_0^1\int_0^1\int_{\gamma_{\frac{k_1}{\ell_4},\tau_1^{\ell_4}\tau_2^{\ell_5}}}
\int_{\gamma_{\frac{k_2}{\ell_7},\tau_1^{\ell_6}\tau_2^{\ell_7}}}
\int_{\gamma_{\frac{k_2}{\ell_5}}}
\int_{\gamma_{\frac{k_1}{\ell_6}}}\right.\\
\left.\times \mathcal{A}_1^{k_1\ell_1}C_{k_2,\ell_2,\ell_3}(\omega)(\mathcal{A}_1,\mathcal{A}_2,m_1)\mathbb{K}_{\overline{\ell},k_1,k_2}(\boldsymbol{t},\boldsymbol{\xi},\tau_1^{\ell_4}\tau_2^{\ell_5},\tau_1^{\ell_6}\tau_2^{\ell_7})dt_1dt_2d\xi_1d\xi_2d\xi_3d\xi_4 R_{\underline{\ell}}(im_1)dm_1\right\|_{(\boldsymbol{\nu},\beta,\mu,\boldsymbol{k},\rho)}\\
\le C_5 \mathfrak{A}_{\underline{\ell}}\left\|\omega\right\|_{(\boldsymbol{\nu},\beta,\mu,\boldsymbol{k},\rho)},
\end{multline*}
valid for every $\omega\in F^{\boldsymbol{d}}_{(\boldsymbol{\nu},\beta,\mu,\boldsymbol{k},\rho)}$.
\end{lemma}
\begin{proof}
The proof follows an analogous argument as that for Lemma~\ref{lema660}. Here,
$$|C_{k_2,\ell_2,\ell_3}(\omega)(\mathcal{A}_1,\mathcal{A}_2,m_1)|\le \frac{1}{\Gamma\left(\frac{\ell_2}{k_2}\right)}\frac{M_{\rho,\boldsymbol{\nu},\boldsymbol{k}}}{(1+|m_1|)^{\mu}}e^{-\beta|m_1|}\left\|\omega\right\|_{(\boldsymbol{\nu},\beta,\mu,\boldsymbol{k},\rho)}|\mathcal{A}_2|^{\ell_2+k_2\ell_3}B(\frac{\ell_2}{k_2},\ell_3).$$
In addition to this, $\tilde{M}_3$ and $\tilde{M}_4$ are substituted by 
$$\check{M}_3=\int_0^1(1-t_1)^{\frac{\ell_4}{k_1}k_1\ell_1-1}t_1^{\frac{\ell_6}{k_1}(k_2\ell_3+\ell_2)-1}dt_1<\infty,$$
and
$$\check{M}_4=\int_0^1(1-t_2)^{\frac{\ell_5}{k_2}k_1\ell_1-1}t_2^{\frac{\ell_7}{k_2}(k_2\ell_3+\ell_2)-1}dt_2<\infty,$$
being finite numbers in virtue of assumption (v) made on the elements in $\mathcal{A}$. The conclusion follows from analogous computations as in Lemma~\ref{lema660}.
\end{proof}

The proof of the following result combines the proof of Lemma~\ref{lema660} and Lemma~\ref{lema661}, so we omit it.

\begin{lemma}\label{lema662}
Let $\underline{\ell}=(\ell_0,\ldots,\ell_7)\in\mathcal{A}$ with $\ell_4,\ell_7\ge 2$ and $\ell_5,\ell_6\ge1$ with $\ell_0\ge1$ and $\ell_2\ge1$. Then, there exists $C_6>0$ such that 
\begin{multline*}
\left\|\int_{-\infty}^{\infty}\frac{1}{P_{m,1}(\tau_1)P_{m,2}(\tau_2)}A_{\underline{\ell}}(m-m_1) 
\int_0^1\int_0^1\int_{\gamma_{\frac{k_1}{\ell_4},\tau_1^{\ell_4}\tau_2^{\ell_5}}}
\int_{\gamma_{\frac{k_2}{\ell_7},\tau_1^{\ell_6}\tau_2^{\ell_7}}}
\int_{\gamma_{\frac{k_2}{\ell_5}}}
\int_{\gamma_{\frac{k_1}{\ell_6}}}\right.\\
\left.\times C_{k_1,k_2,\ell_0,\ell_2,\ell_1,\ell_3}(\omega)(\mathcal{A}_1,\mathcal{A}_2,m_1)\mathbb{K}_{\overline{\ell},k_1,k_2}(\boldsymbol{t},\boldsymbol{\xi},\tau_1^{\ell_4}\tau_2^{\ell_5},\tau_1^{\ell_6}\tau_2^{\ell_7})dt_1dt_2d\xi_1d\xi_2d\xi_3d\xi_4 R_{\underline{\ell}}(im_1)dm_1\right\|_{(\boldsymbol{\nu},\beta,\mu,\boldsymbol{k},\rho)}\\
\le C_6 \mathfrak{A}_{\underline{\ell}}\left\|\omega\right\|_{(\boldsymbol{\nu},\beta,\mu,\boldsymbol{k},\rho)},
\end{multline*}
valid for every $\omega\in F^{\boldsymbol{d}}_{(\boldsymbol{\nu},\beta,\mu,\boldsymbol{k},\rho)}$.
\end{lemma}

Upper estimates for the nonlinear part can be achieved following the same argument as in Lemma 7,~\cite{ma24}, and taking into account the assumptions (\ref{e112}) and (\ref{e142}).

\begin{lemma}\label{lema748}
There exists a constant $C_7>0$ such that 
\begin{multline*}
\left\|\frac{1}{P_{m,1}(\tau_1)P_{m,2}(\tau_2)}\int_{-\infty}^{\infty}\left(\tau_1^{k_1}\tau_2^{k_2}\int_0^{\tau_1^{k_1}}\int_0^{\tau_2^{k_2}} \omega_1((\tau_1^{k_1}-s_1)^{1/k_1},(\tau_2^{k_2}-s_2)^{1/k_2},m-m_1)\right.\right.\\
\left.\left.\times\tilde{Q}_1(i(m-m_1))\omega_2(s_1^{1/k_1},s_2^{1/k_2},m_1)\tilde{Q}_2(im_1)\frac{ds_1ds_2}{(\tau_1^{k_1}-s_1)s_1(\tau_2^{k_2}-s_2)s_2}\right)dm_1\right\|_{(\boldsymbol{\nu},\beta,\mu,\boldsymbol{k},\rho)}\\
\le C_7\left\|\omega_1\right\|_{(\boldsymbol{\nu},\beta,\mu,\boldsymbol{k},\rho)}\left\|\omega_2\right\|_{(\boldsymbol{\nu},\beta,\mu,\boldsymbol{k},\rho)},
\end{multline*}
for every $\omega_1,\omega_2\in F^{\boldsymbol{d}}_{(\boldsymbol{\nu},\beta,\mu,\boldsymbol{k},\rho)}$.
\end{lemma}

Finally, the forcing term can be studied as in Lemma 8,~\cite{ma24}. It admits the following estimates.

\begin{lemma}\label{lema749}
There exists a constant $C_8>0$ such that
$$\left\|\frac{1}{P_{m,1}(\tau_1)P_{m,2}(\tau_2)}\sum_{\underline{j}=(j_1,j_2)\in J}\mathcal{F}_j(m)\tau_1^{j_1}\tau_2^{j_2}\right\|_{(\boldsymbol{\nu},\beta,\mu,\boldsymbol{k},\rho)}\le C_8 \max_{\underline{j}\in J}\left\|\mathcal{F}_{\underline{j}}\right\|_{(\beta,\mu)}.$$
\end{lemma}

We are in position to solve the second auxiliary problem (\ref{eaux2}).

\begin{prop}\label{prop767}
Let $0<\rho<1$ and unbounded sectors $S_{d_j}$ edged at 0 with bisecting directions $d_j\in\R$ for $j\in\{1,2\}$, fulfilling the conditions (\ref{e499b}) and (\ref{e500b}). There exist $\mathcal{M},\mathcal{F},\mathcal{C}>0$ such that if
$$\mathfrak{A}_{\underline{\ell}}\le \mathcal{M},\quad \mathfrak{F}_{\underline{j}}\le \mathcal{F},\quad |\tilde{c}|\le\mathcal{C}$$
for all $\underline{\ell}\in\mathcal{A}$ and $\underline{j}\in J$, and for $\varpi>0$ well chosen, the equation (\ref{eaux2}) admits a unique solution $\omega^{\boldsymbol{d}}\in F^{\boldsymbol{d}}_{(\boldsymbol{\nu},\beta,\mu,\boldsymbol{k},\rho)}$ with $\left\|\omega^{\boldsymbol{d}}\right\|_{(\boldsymbol{\nu},\beta,\mu,\boldsymbol{k},\rho)}\le \varpi$.
\end{prop}

\begin{proof}
Let $\mathcal{M},\mathcal{F},\mathcal{C}>0$ together with $\varpi$ be chosen in a way that
\begin{multline*}
\frac{1}{(2\pi)^{1/2}}\mathcal{M}\varpi\left[C_1\sum_{\substack{\underline{\ell}=(\ell_0,\ldots,\ell_7)\in\mathcal{A}\\ (\ell_4,\ell_5,\ell_6,\ell_7)=(1,0,0,1)\\ \ell_0=0,\ell_2=0}}k_1^{\ell_1}k_2^{\ell_3}+C_3\sum_{\substack{\underline{\ell}=(\ell_0,\ldots,\ell_7)\in\mathcal{A}\\ (\ell_4,\ell_5,\ell_6,\ell_7)\neq(1,0,0,1)\\ \ell_0=0,\ell_2=0}}k_1^{\ell_1}k_2^{\ell_3}+C_4\sum_{\substack{\underline{\ell}=(\ell_0,\ldots,\ell_7)\in\mathcal{A}\\ (\ell_4,\ell_5,\ell_6,\ell_7)\neq(1,0,0,1)\\ \ell_0\ge1,\ell_2=0}}k_2^{\ell_3} \right.\\
\left.+C_5\sum_{\substack{\underline{\ell}=(\ell_0,\ldots,\ell_7)\in\mathcal{A}\\ (\ell_4,\ell_5,\ell_6,\ell_7)\neq(1,0,0,1)\\ \ell_0=0,\ell_2\ge1}}k_1^{\ell_1}+\hbox{card}(\mathcal{A})(C_{2,1}+C_{2,2}+C_2+C_6)\right]+\frac{1}{(2\pi)^{1/2}}\mathcal{C}C_7\varpi^2+C_8\mathcal{F}\le \varpi.
\end{multline*}
Then, in virtue of Lemma~\ref{lema478}-Lemma~\ref{lema749}, for every $\omega\in F^{\boldsymbol{d}}_{(\boldsymbol{\nu},\beta,\mu,\boldsymbol{k},\rho)}$ with $\left\|\omega\right\|_{(\boldsymbol{\nu},\beta,\mu,\boldsymbol{k},\rho)}\le \varpi$, we get that $\mathcal{H}(\omega)\in F^{\boldsymbol{d}}_{(\boldsymbol{\nu},\beta,\mu,\boldsymbol{k},\rho)}$, with $\left\|\mathcal{H}(\omega)\right\|_{(\boldsymbol{\nu},\beta,\mu,\boldsymbol{k},\rho)}\le \varpi$. Therefore, the mapping $\mathcal{H}$ defined in (\ref{eaux3}) induces a map from $B(0,\varpi)\subseteq F^{\boldsymbol{d}}_{(\boldsymbol{\nu},\beta,\mu,\boldsymbol{k},\rho)}$ into itself, where $B(0,\varpi)$ stands for the closed ball centered at 0 and radius $\varpi$ in the space $F^{\boldsymbol{d}}_{(\boldsymbol{\nu},\beta,\mu,\boldsymbol{k},\rho)}$. 

We reduce $\varpi,\mathcal{M},\mathcal{F},\mathcal{C}>0$ such that
\begin{multline*}
\frac{1}{(2\pi)^{1/2}}\mathcal{M}\left[C_1\sum_{\substack{\underline{\ell}=(\ell_0,\ldots,\ell_7)\in\mathcal{A}\\ (\ell_4,\ell_5,\ell_6,\ell_7)=(1,0,0,1)\\ \ell_0=0,\ell_2=0}}k_1^{\ell_1}k_2^{\ell_3}+C_3\sum_{\substack{\underline{\ell}=(\ell_0,\ldots,\ell_7)\in\mathcal{A}\\ (\ell_4,\ell_5,\ell_6,\ell_7)\neq(1,0,0,1)\\ \ell_0=0,\ell_2=0}}k_1^{\ell_1}k_2^{\ell_3}+C_4\sum_{\substack{\underline{\ell}=(\ell_0,\ldots,\ell_7)\in\mathcal{A}\\ (\ell_4,\ell_5,\ell_6,\ell_7)\neq(1,0,0,1)\\ \ell_0\ge1,\ell_2=0}}k_2^{\ell_3} \right.\\
\left.+C_5\sum_{\substack{\underline{\ell}=(\ell_0,\ldots,\ell_7)\in\mathcal{A}\\ (\ell_4,\ell_5,\ell_6,\ell_7)\neq(1,0,0,1)\\ \ell_0=0,\ell_2\ge1}}k_1^{\ell_1}+\hbox{card}(\mathcal{A})(C_{2,1}+C_{2,2}+C_2+C_6)\right]+\frac{2}{(2\pi)^{1/2}}\mathcal{C}C_7\varpi \le \frac{1}{2}.
\end{multline*}

Let $\omega_1,\omega_2\in F^{\boldsymbol{d}}_{(\boldsymbol{\nu},\beta,\mu,\boldsymbol{k},\rho)}$ with $\left\|\omega_j\right\|_{(\boldsymbol{\nu},\beta,\mu,\boldsymbol{k},\rho)}\le \varpi$ for $j=1,2$. 

Concerning the nonlinear part of $\mathcal{H}(\omega_1)-\mathcal{H}(\omega_2)$ we observe that if we denote $d_{1,j}=\omega_j((\tau_1^{k_1}-s_1)^{1/k_1},(\tau_2^{k_2}-s_2)^{1/k_2},m-m_1)$ for $j=1,2$, $d_2=\tilde{Q}_1(i(m-m_1))$, $d_{3,j}=\omega_j(s_1^{1/k_1},s_2^{1/k_2},m_1)$ for $j=1,2$, and $d_4=\tilde{Q}_2(im_1)$, then one has
\begin{equation}\label{e888}
d_{1,1}d_2d_{3,1}d_4-d_{1,2}d_2d_{3,2}d_4=d_2\left[d_{1,1}-d_{1,2}\right]d_4d_{3,1}+d_{2}d_{1,2}d_4\left[d_{3,1}-d_{3,2}\right].
\end{equation}
Therefore, in virtue of Lemma~\ref{lema478}-Lemma~\ref{lema662}, together with (\ref{e888}) and Lemma~\ref{lema748}, one has that 
$$\left\|\mathcal{H}(\omega_1)-\mathcal{H}(\omega_2)\right\|_{(\boldsymbol{\nu},\beta,\mu,\boldsymbol{k},\rho)}\le\frac{1}{2}\left\|\omega_1-\omega_2\right\|_{(\boldsymbol{\nu},\beta,\mu,\boldsymbol{k},\rho)}.$$

The classical contractive mapping theorem guarantees the existence of a unique $\omega^{\boldsymbol{d}}\in F^{\boldsymbol{d}}_{(\boldsymbol{\nu},\beta,\mu,\boldsymbol{k},\rho)}$ such that $\mathcal{H}(\omega^{\boldsymbol{d}})=\omega^{\boldsymbol{d}}$. As a conclusion, $\omega^{\boldsymbol{d}}$, with $\left\|\omega^{\boldsymbol{d}}\right\|_{(\boldsymbol{\nu},\beta,\mu,\boldsymbol{k},\rho)}\le \varpi$, is a solution of (\ref{eaux2}).
\end{proof}

\section{Main result}\label{secpral}

In this last section, we state the main result of the present work. We preserve that definitions and assumptions made in Section~\ref{sec3} regarding the elements constructing the main problem under study (\ref{epral}), under null initial data $u(t_1,0,z)\equiv u(0,t_2,z)\equiv 0$. 

The main result states the construction of a formal solution to (\ref{epral}). In addition to this, we prove that the two-dimensional Laplace transform of the solution to (\ref{eaux2}) is a function admitting strong asymptotic expansion in a polysector, as defined in Section~\ref{sec877}.

\begin{theo}\label{teopral}
Let $0<\rho<1$ and unbounded sectors $S_{d_j}$ edged at 0 with bisecting directions $d_j\in\R$ for $j\in\{1,2\}$ subjected to the constraints (\ref{e499b}) and (\ref{e500b}) (from Proposition~\ref{prop426}). Assume that $\mathfrak{A}_{\underline{\ell}}$ for $\underline{\ell}\in \mathcal{A}$, $\mathfrak{F}_{\underline{j}}$ for $\underline{j}\in J$ together with $\tilde{c}\in\C^{\star}$ be small enough as determined in Proposition~\ref{prop767}. Let $\omega^{\boldsymbol{d}}$ be the solution of (\ref{eaux2}) determined in Proposition~\ref{prop767}. 
\begin{itemize}
\item[(i)] Let us write
$$(\tau_1,\tau_2)\mapsto \omega^{\boldsymbol{d}}(\tau_1,\tau_2,m)=\sum_{p_1,p_2\ge1}\omega_{p_1,p_2}^{\boldsymbol{d}}(m)\frac{\tau_1^{p_1}}{\Gamma\left(\frac{p_1}{k_1}\right)}\frac{\tau_2^{p_2}}{\Gamma\left(\frac{p_2}{k_2}\right)}\in E_{(\beta,\mu)}\{\tau_1,\tau_2\}$$
for Taylor's expansion of the function $(\tau_1,\tau_2)\mapsto \omega^{\boldsymbol{d}}(\tau_1,\tau_2,m)$, as a function defined in $D(0,\rho)^2$ with values in $E_{(\beta,\mu)}$. Then, the formal power series
\begin{equation}\label{e792}
\hat{u}(t_1,t_2,z):=\sum_{p_1,p_2\ge1}\mathcal{F}^{-1}(m\mapsto\omega_{p_1,p_2}^{\boldsymbol{d}}(m))(z)t_1^{p_1}t_2^{p_2}
\end{equation}
is a formal solution to (\ref{epral}) in the variables $t_1$, $t_2$ whose coefficients belong to the Banach space $\mathcal{O}_b(H_{\beta'})$ of bounded holomorphic functions on $H_{\beta'}$ endowed with the sup norm, for any given $0<\beta'<\beta$.
\item[(ii)] The function $(t_1,t_2)\mapsto\mathcal{L}_{\boldsymbol{k}}^{\boldsymbol{d}}(\omega^{\boldsymbol{d}})(\boldsymbol{t})$ is holomorphic on the product of two bounded sectors with bisecting directions $d_1$ and $d_2$, resp., say $\tilde{S}_{d_1}\times\tilde{S}_{d_2}$, with values in the Banach space $E_{(\beta,\mu)}$, and admits strong asymptotic expansion of Gevrey order $\boldsymbol{k}$ on $\tilde{S}_{d_1}\times\tilde{S}_{d_2}$.
\end{itemize}
\end{theo}
\begin{proof}

The function $(\tau_1,\tau_2)\mapsto \omega^{\boldsymbol{d}}(\tau_1,\tau_2,m)$ is holomorphic on $D(0,\rho)^2$ with values in $E_{(\beta,\mu)}$ as given in (i) of Theorem~\ref{teopral}. In view of Proposition~\ref{prop367}, the formal power series
\begin{equation}\label{e923}
\sum_{p_1,p_2\ge1}\omega_{p_1,p_2}^{\boldsymbol{d}}(m)t_1^{p_1}t_2^{p_2}
\end{equation}
is a formal solution to (\ref{eaux1}). The first statement follows from the properties of inverse Fourier transform, which determine that $\hat{u}$ as defined in (\ref{e792}) is a formal solution of (\ref{epral}).

The second statement is a consequence of the results on strong asymptotic expansions recalled in Section~\ref{sec877}. Indeed, for every $m\in\R$, the function $(\tau_1,\tau_2)\mapsto\omega^{\boldsymbol{d}}(\tau_1,\tau_2,m)$ is a holomorphic function in $D(0,\rho)^2$ with values in $E_{(\beta,\mu)}$, which can be analytically extended to $(S_{d_1}\cup D(0,\rho))\times(S_{d_2}\cup D(0,\rho))$ with exponential growth at infinity of order $(k_1,k_2)$:
\begin{multline*}
\left\|(\tau_1,\tau_2)\mapsto\omega^{\boldsymbol{d}}(\tau_1,\tau_2,m)\right\|_{(\beta,\mu)}\\
\le \left\|(\tau_1,\tau_2,m)\mapsto\omega^{\boldsymbol{d}}(\tau_1,\tau_2,m)\right\|_{(\boldsymbol{\nu},\beta,\mu,\boldsymbol{k},\rho)}\frac{|\tau_1|}{1+|\tau_1|^{2k_1}}\frac{|\tau_2|}{1+|\tau_2|^{2k_2}}\exp(\nu_1|\tau_1|^{k_1}+\nu_2|\tau_2|^{k_2}),
\end{multline*}
for every $(\tau_1,\tau_2)\in (S_{d_1}\cup D(0,\rho))\times(S_{d_2}\cup D(0,\rho))$. 

The $E_{(\beta,\mu)}$-valued family associated to the strong asymptotic expansion is given by
$$\mathcal{F}=\{ \{h_{p,q}\}_{p,q\ge 0}, \{h^{(1)}_{p}(t_2)\}_{p\ge0}, \{h^{(2)}_{q}(t_1)\}_{q\ge0}\},$$
where for all integers $p,q\ge1$
$$h_{p}^{(1)}(t_2):=\frac{\Gamma\left(\frac{p}{k_1}\right)}{p!}k_2\int_{L_{d_2}}(\partial_{\tau_1}^p\omega^{\boldsymbol{d}})(0,\tau_2,m)\exp\left(-\left(\frac{\tau_2}{t_2}\right)^{k_2}\right)\frac{d\tau_2}{\tau_2},$$
$$h_{q}^{(2)}(t_1):=\frac{\Gamma\left(\frac{q}{k_2}\right)}{q!}k_1\int_{L_{d_1}}(\partial_{\tau_2}^q\omega^{\boldsymbol{d}})(\tau_1,0,m)\exp\left(-\left(\frac{\tau_1}{t_1}\right)^{k_1}\right)\frac{d\tau_1}{\tau_1},$$
and
$$h_{p,q}:=\frac{\Gamma\left(\frac{p}{k_1}\right)\Gamma\left(\frac{q}{k_2}\right)}{p!q!}(\partial_{\tau_1}^p\partial_{\tau_2}^q\omega^{\boldsymbol{d}})(0,0,m),$$
along with $h_0^{(1)}(t_2)\equiv 0$, $h_0^{(2)}(t_1)\equiv 0$ and $h_{0,q}\equiv 0$ for all $q\ge0$, $h_{p,0}\equiv 0$ for all $p\ge0$.

\end{proof}

\vspace{0.3cm}

\textbf{Remark:} Although the formal series (\ref{e792}) obtained as a Fourier transform of the formal series (\ref{e923}) solves the main equation (\ref{epral}), the inverse Fourier transform
$$u^{\boldsymbol{d}}(\boldsymbol{t},z):=\mathcal{F}^{-1}\left(m\mapsto\mathcal{L}_{\boldsymbol{k}}^{\boldsymbol{d}}(\omega^{\boldsymbol{d}})(t)\right)(z)$$
of the map defined in Theorem~\ref{teopral} ii) which represents a bounded holomorphic map on the product $\tilde{S}_{d_1}\times \tilde{S}_{d_2}\times H_{\beta'}$ for any $0<\beta'<\beta$, does in general, not solve the same equation (\ref{epral}). This phenomenon is similar to the one observed in our previous works~\cite{lama4,ma24}. Indeed, the action of both infinite order differential operators $\cosh(\alpha_j(t_j^{k_j+1}\partial_{t_j})^{\Delta_j})$ given in (\ref{e220}) is ill-defined on $u^{\boldsymbol{d}}(\boldsymbol{t},z)$ since exactly one of the maps $\tau_j\mapsto \exp\left(\pm\alpha_j(k_j\tau_j^{k_j})^{\Delta_j}\right)$ has exponential growth of order $k_j\Delta_j(>k_j)$ where a growth rate of at most $k_j$ is asked to ensure its $(k_1,k_2)-$Laplace transformability along $L_{d_1}\times L_{d_2}$.

\section{Appendix I: Review on function spaces and operators}

In this last section, we recall the main definitions and results on different function spaces and operators involved in the main problem under study. Most of them have been used in previous research, so we omit details on the proofs, the reader being addressed to appropriate references.

\subsection{Inverse Fourier transform}\label{secanexo1}
We fix $\beta>0$ and $\mu>1$ in the whole section. Given any continuous function $h:\R\to\C$, such that 
\begin{equation}\label{e152}
\sup_{m\in\R}|h(m)|(1+|m|)^{\mu}\exp(\beta|m|)<\infty,
\end{equation}
one can define its associated inverse Fourier transform by
$$ \mathcal{F}^{-1}(h)(x)=\frac{1}{(2\pi)^{1/2}}\int_{-\infty}^{\infty}h(m)\exp(imx)dm,\qquad x\in\R,$$
which determines a holomorphic and bounded function in $H_{\beta'}=\{z\in\C:|\hbox{Im}(z)|<\beta'\}$, for every fixed $0<\beta'<\beta$. 

\begin{defin}
The set of continuous functions $h:\R\to\C$ satisfying (\ref{e152}) defines a Banach space $(E_{(\beta,\mu)},\left\|\cdot\right\|_{(\beta,\mu)})$, where  
$$\left\|h\right\|_{(\beta,\mu)}:=\sup_{m\in\R}|h(m)|(1+|m|)^{\mu}\exp(\beta|m|).$$
\end{defin}

\begin{prop}[Proposition 7,~\cite{lama}]\label{prop164}
For all $h\in E_{(\beta,\mu)}$, the following statements hold:
\begin{itemize}
\item The function $\R\ni m\mapsto \varphi(m)=im h(m)$ belongs to $E_{(\beta,\mu-1)}$. In addition to this, $\partial_z\mathcal{F}^{-1}(h)(z)=\mathcal{F}^{-1}(\varphi)(z)$, for every $z\in H_{\beta}$.
\item Let $g\in E_{(\beta,\mu)}$. The convolution product of $h$ and $g$, defined by
$$\R\ni m\mapsto (h\star g)(m)=\frac{1}{(2\pi)^{1/2}}\int_{-\infty}^{\infty}h(m-m_1)g(m_1)dm_1$$
belongs to $E_{(\beta,\mu)}$ and it holds that
$$\mathcal{F}^{-1}(h)(z)\mathcal{F}^{-1}(g)(z)=\mathcal{F}^{-1}(h\star g)(z),\qquad z\in H_{\beta}.$$
Moreover, there exists $C=C(\mu)>0$ such that
$$\left\|h\star g\right\|_{(\beta,\mu)}\le C \left\|h\right\|_{(\beta,\mu)}\left\|g\right\|_{(\beta,\mu)}.$$
\end{itemize}
\end{prop}

\subsection{Banach space of functions with exponential growth in two variables}\label{secanexo2}

In this section, we define a Banach space of functions defined in three variables. The first two are complex variables, defined on a neighborhood of the origin in $\C^2$ which can be extended in each of the variables along an unbounded sector with exponential growth. Regarding the third variable, the functions are defined on the real line, and satisfy exponential decreasement at infinity analogously to the functions described in the previous section. This Banach space is an adaptation of that considered in~\cite{lama} to two time variables.

\begin{defin}
Let $d_1,d_2\in\R$ and let $S_{d_j}\subseteq\C$ be an unbounded sector with bisecting direction $d_j$ and vertex at the origin, for $j=1,2$. Let $\nu_1,\nu_2,\rho>0$. Let $k_1,k_2\ge1$ be integers, and let $\beta>0$ and $\mu>1$. The set $F^{\boldsymbol{d}}_{(\boldsymbol{\nu},\beta,\mu,\boldsymbol{k},\rho)}$ consists of all continuous functions $(\tau_1,\tau_2,m)\mapsto h(\tau_1,\tau_2,m)$ defined on $(S_{d_1}\cup D(0,\rho))\times (S_{d_2}\cup D(0,\rho))\times \R$, holomorphic with respect to $(\tau_1,\tau_2)$ on $(S_{d_1}\cup D(0,\rho))\times (S_{d_2}\cup D(0,\rho))$ such that
\begin{multline*}
\left\|h(\tau_1,\tau_2,m)\right\|_{(\boldsymbol{\nu},\beta,\mu,\boldsymbol{k},\rho)}:=\sup_{
\substack{\tau_j\in S_{d_j}\cup D(0,\rho),j=1,2\\ m\in\R}}(1+|m|)^{\mu}e^{\beta|m|}\frac{1+|\tau_1|^{2k_1}}{|\tau_1|}\frac{1+|\tau_2|^{2k_2}}{|\tau_2|}\\
\times\exp\left(-\nu_1|\tau_1|^{k_1}-\nu_2|\tau_2|^{k_2}\right)|h(\boldsymbol{\tau},m)|<\infty.
\end{multline*}
The pair $(F^{\boldsymbol{d}}_{(\boldsymbol{\nu},\beta,\mu,\boldsymbol{k},\rho)},\left\|\cdot\right\|_{(\boldsymbol{\nu},\beta,\mu,\boldsymbol{k},\rho)})$ is a complex Banach space.
\end{defin}

The following result is a direct consequence of the previous definition.

\begin{lemma}\label{lema7}
Let $d_1,d_2\in\R$ and let $S_{d_j}\subseteq\C$ be an unbounded sector with bisecting direction $d_j$ and vertex at the origin, for $j=1,2$. Let $\nu_1,\nu_2,\rho>0$. Let $k_1,k_2\ge1$ be integers, and let $\beta>0$ and $\mu>1$. There exists $M_{\rho,\boldsymbol{\nu},\boldsymbol{k}}>0$ only depending on $\rho,\boldsymbol{\nu}, \boldsymbol{k}$ such that for every $f\in F^{\boldsymbol{d}}_{(\boldsymbol{\nu},\beta,\mu,\boldsymbol{k},\rho)}$ one has
$$\sup_{\boldsymbol{\tau}\in D(0,\rho)^2}|f(\tau_1,\tau_2,m)|\le M_{\rho,\boldsymbol{\nu},\boldsymbol{k}}\frac{1}{(1+|m|)^{\mu}}e^{-\beta|m|} \left\|f(\tau_1,\tau_2,m)\right\|_{(\boldsymbol{\nu},\beta,\mu,\boldsymbol{k},\rho)},$$
for all $m\in\R$.
\end{lemma}
\begin{proof}
Define the constant $M_{\rho,\boldsymbol{\nu},\boldsymbol{k}}:=\rho^2\exp(\nu_1\rho^{k_1}+\nu_2\rho^{k_2})$ and apply the definition of an element in $F^{\boldsymbol{d}}_{(\boldsymbol{\nu},\beta,\mu,\boldsymbol{k},\rho)}$.
\end{proof}

\section{Appendix II: Majima's asymptotic expansions}\label{sec877}

In this appendix, we recall the notion of Majima's asymptotic expansions of functions in several complex variables with values in a Banach space $(\mathbb{E},\left\|\cdot\right\|_{\mathbb{E}})$. This concept was put forward by H. Majima~\cite{majima1,majima2}, generalizing the concept of asymptotic expansions of H. Poincar\'e to the several variable setting. We recall its definition in the two variable case, which is considered in the present study.

\begin{defin}
Let $S_j$ be a sector in the complex domain for $j=1,2$. A function $f\in\mathcal{O}(S_1\times S_2,\mathbb{E})$ admits a strong asymptotic expansion in $S_1\times S_2$ (at the origin) if there exists a family 
$$\mathcal{F}=\{\{h^{(1)}_{p}(t_2)\}_{p\ge0}, \{h^{(2)}_{q}(t_1)\}_{q\ge0}, \{h_{p,q}\}_{\substack{p,q\ge 0}}\},$$
where $h_{p,q}\in\mathbb{E}$, $h^{(1)}_{p}\in\mathcal{O}(S_2,\mathbb{E})$, $h^{(2)}_{q}\in\mathcal{O}(S_1,\mathbb{E})$ for all integers $p,q\ge 0$ such that for any integers $N_1\ge 1$, $N_2\ge1$, and for every bounded subpolysector $\boldsymbol{T}=T_1\times T_2\prec S_1\times S_2$, there exists $c=c(N_1,N_2,\boldsymbol{T})>0$, with
$$\left\|f(t_1,t_2)-\sum_{p=0}^{N_1-1}h^{(1)}_{p}(t_2)t_1^{p}-\sum_{q=0}^{N_2-1}h^{(2)}_{q}(t_1)t_2^{q}+\sum_{p=0}^{N_1-1}\sum_{q=0}^{N_2-1}h_{p,q}t_1^{p}t_2^{q}\right\|_{\mathbb{E}}\le c|t_1|^{N_1}|t_2|^{N_2},$$
for every $(t_1,t_2)\in \boldsymbol{T}$.
\end{defin}
Following Y. Haraoka in~\cite{haraoka}, one also has the notion of Gevrey strong asymptotic expansions.
\begin{defin}
A function admits strong Gevrey asymptotic expansion of order $\boldsymbol{k}=(k_1,k_2)\in\mathbb{N}^2$ in a polysector if it admits strong asymptotic expansion in that polysector, and the constant $c=c(\boldsymbol{N},\boldsymbol{T})$ is of the form
$$c(\boldsymbol{N},\boldsymbol{T})=cA^{N_1+N_2}\Gamma\left(\frac{N_1}{k_1}\right)\Gamma\left(\frac{N_2}{k_2}\right),$$
for some $c=c(\boldsymbol{T})>0$ and $A=A(\boldsymbol{T})>0$.
\end{defin}

In the literature one can find advances adapting one-dimensional results in this context, such as Borel-Ritt theorem in~\cite{gasa,sa1}, the comparison with other ways to define asymptotic expansions in several complex variables in~\cite{hesa}. 



Let $\omega\in\mathcal{O}(D(0,\rho)^2,\mathbb{E})$, for some $\rho>0$, with $\omega(\tau_1,0)\equiv\omega(0,\tau_2)\equiv0$. Let us write its Taylor expansion at the origin in the form
$$\omega(\tau_1,\tau_2)=\sum_{p_1,p_2\ge1}\omega_{p_1,p_2}\frac{\tau_1^{p_1}}{\Gamma\left(\frac{p_1}{k_1}\right)}\frac{\tau_2^{p_2}}{\Gamma\left(\frac{p_2}{k_2}\right)}\in\mathbb{E}[[\tau_1,\tau_2]].$$
Assume that $\omega$ can be analytically continued to the product of two infinite sectors of $S_{d_1}\times S_{d_2}$, where $S_{d_j}$ is an infinite sector with bisecting direction $d_j\in\R$, for $j=1,2$. Assume moreover that this analytic continuation admits exponential growth of order $(k_1,k_2)\in\mathbb{N}^2$ in $S_{d_1}\times S_{d_2}$, i.e. there exist $C,K_1,K_2>0$ such that
\begin{equation}\label{e911}
\left\|\omega(\tau_1,\tau_2)\right\|_{\mathbb{E}}\le C\exp\left(K_1|\tau_1|^{k_1}+K_2|\tau_2|^{k_2}\right)|\tau_1\tau_2|,
\end{equation}
for every $\tau_j\in D(0,\rho)\cup S_{d_j}$, $j=1,2$. Then, the Laplace transform of order $(k_1,k_2)$ in the multidirection $\boldsymbol{d}=(d_1,d_2)$ of $\omega(\boldsymbol{\tau})$ can be defined by 
$$\mathcal{L}_{\boldsymbol{k}}^{\boldsymbol{d}}(\omega)(\boldsymbol{t})=k_1k_2\int_{L_{d_1}}\int_{L_{d_2}}\omega(\tau_1,\tau_2)\exp\left(-\left(\frac{\tau_1}{t_1}\right)^{k_1}-\left(\frac{\tau_2}{t_2}\right)^{k_2}\right)\frac{d\tau_2}{\tau_2}\frac{d\tau_1}{\tau_1},$$
where $L_{d_j}=[0,\infty)e^{id_j}$ for $j=1,2$. 

It turns out that $\mathcal{L}_{\boldsymbol{k}}^{\boldsymbol{d}}(\omega)$ is a bounded holomorphic function on the product of two bounded sectors $S_{d_1,\theta_1,R_1}\times S_{d_2,\theta_2,R_2}$ with values in $\mathbb{E}$, where
$$S_{d_j,\theta_j,R_j}=\left\{t_j\in\C^{\star}:|t_j|<R_j,|d_j-\hbox{arg}(t_j)|<\frac{\theta_j}{2}\right\},\qquad j=1,2,$$
for $\pi/k_j<\theta_j<\pi/k_j+2\delta_j$, where $2\delta_j$ is the opening of $S_{d_j}$ and $R_j>0$ is some positive number. The previous assertion comes from the fact that the initial integration ray $d_j$ can be rotated inside the arguments of $S_{d_j}$, giving rise to a unique holomorphic function by analytic continuation in the wider sector $S_{d_j,\theta_j,R_j}$. In addition to this, $\mathcal{L}_{\boldsymbol{k}}^{\boldsymbol{d}}(\omega)$ admits Gevrey strong asymptotic expansion of order $(k_1,k_2)$ in $S_{d_1,\theta_1,R_1}\times S_{d_2,\theta_2,R_2}$, with associated family
\begin{equation}\label{e1051}
\mathcal{F}=\left\{\{h_{p}^{(1)}(t_2)\}_{p\ge0},\{h_{q}^{(2)}(t_1)\}_{q\ge0}, \{h_{p,q}\}_{p,q\ge0}\right\},
\end{equation}
where for all integers $p\ge1$, $q\ge 1$, 
$$h_{p}^{(1)}(t_2):=\frac{\Gamma\left(\frac{p}{k_1}\right)}{p!}k_2\int_{L_{d_2}}(\partial_{\tau_1}^p\omega)(0,\tau_2)\exp\left(-\left(\frac{\tau_2}{t_2}\right)^{k_2}\right)\frac{d\tau_2}{\tau_2},
$$
$$h_{q}^{(2)}(t_1):=\frac{\Gamma\left(\frac{q}{k_2}\right)}{q!}k_1\int_{L_{d_1}}(\partial_{\tau_2}^q\omega)(\tau_1,0)\exp\left(-\left(\frac{\tau_1}{t_1}\right)^{k_1}\right)\frac{d\tau_1}{\tau_1},
$$
and
$$h_{p,q}:=\frac{\Gamma\left(\frac{p}{k_1}\right)\Gamma\left(\frac{q}{k_2}\right)}{p!q!}(\partial_{\tau_1}^p\partial_{\tau_2}^q\omega)(0,0),
$$
together with $h_0^{(1)}(t_2)\equiv 0$, $h_0^{(2)}(t_1)\equiv 0$ and $h_{0,q}=h_{p,0}=0$ for all $p,q\ge 0$.

Let us provide some detail on the previous assertion. For given integers $N_1\ge1$, $N_2\ge1$, we consider the two maps 
$$P_{N_1,\tau_2}(\tau_1):=\sum_{p=1}^{N_1-1}\frac{(\partial_{\tau_1}^p\omega)(0,\tau_2)}{p!}\tau_1^p$$
and
$$P_{N_2,\tau_1}(\tau_2):=\sum_{q=1}^{N_2-1}\frac{(\partial_{\tau_2}^q\omega)(\tau_1,0)}{q!}\tau_2^q$$
provided that $\tau_1\in S_{d_1}\cup D(0,\rho)$ and $\tau_2\in S_{d_2}\cup D(0,\rho)$.

For given integers $h_1,h_2\ge0$, taking into account the Cauchy formula in two variables, we have the integral representation given by
$$(\partial_{\tau_2}^{h_2}\partial_{\tau_1}^{h_1}\omega)(u_1,u_2)=\frac{h_1!h_2!}{(2\pi i)^2}\int_{C_{\tilde{R}_1}(u_1)}\int_{C_{\tilde{R}_2}(u_2)}\frac{\omega(\xi_1,\xi_2)d\xi_2d\xi_1}{(\xi_1-u_1)^{h_1+1}(\xi_2-u_2)^{h_2+1}},$$
where $C_{\tilde{R}_j}(u_j)$, $j=1,2$ is any counterclockwise oriented circle centered at $u_j$ with small radius $\tilde{R}_j>0$. In view of the bounds (\ref{e911}), one derives that 
\begin{equation}\label{e970}
\left\|(\partial_{\tau_2}^{h_2}\partial_{\tau_1}^{h_1}\omega)(u_1,u_2)\right\|_{\mathbb{E}}\le \tilde{C}(M_1)^{h_1}(M_2)^{h_2}h_1!h_2!\exp\left(\tilde{K}_1|u_1|^{k_1}+\tilde{K}_2|u_2|^{k_2}\right),
\end{equation}
for some constants $\tilde{C},\tilde{K}_j,M_j>0$ for $j=1,2$, provided that $u_j\in \tilde{S}_{d_j}\cup D(0,\tilde{\rho}_j)$ where $\tilde{S}_{d_j}$ is an unbounded sector edged at 0 with $\tilde{S}_{d_j}\subseteq S_{d_j}$ and smaller opening, and where $0<\tilde{\rho}_j<\rho_j$. Owing to (\ref{e970}), we deduce that all the maps $h_p^{(1)}(t_2)$, $h_{q}^{(2)}(t_1)$ are well defined for all $p,q\ge1$, for $(t_1,t_2)\in T_1\times T_2$, provided that $S_j$ has small enough radius for $j=1,2$. 

Now, we provide integral representations for all quantities
$$H_1(t_1,t_2)=\sum_{p=1}^{N_1-1}h_p^{(1)}(t_2)t_1^p,\quad H_2(t_1,t_2)=\sum_{q=1}^{N_2-1}h_q^{(2)}(t_1)t_2^q,\quad H_3(t_1,t_2)=\sum_{p=1}^{N_1-1}\sum_{q=1}^{N_2-1}h_{p,q}t_1^pt_2^q.$$
Indeed, we observe that
\begin{multline*}
k_1k_2\int_{L_{d_1}}\int_{L_{d_2}}P_{N_1,\tau_2}(\tau_1)\exp\left(-\left(\frac{\tau_1}{t_1}\right)^{k_1}-\left(\frac{\tau_2}{t_2}\right)^{k_2}\right)\frac{d\tau_2}{\tau_2}\frac{d\tau_1}{\tau_1}\\
=\sum_{p=1}^{N_1-1}\left(k_1\int_{L_{d_1}}\tau_1^p\exp\left( -\left(\frac{\tau_1}{t_1}\right)^{k_1}\right)\frac{d\tau_1}{\tau_1}\right)\left(k_2\int_{L_{d_2}}\frac{(\partial_{\tau_1}^{p}\omega)(0,\tau_2)}{p!}\exp\left( -\left(\frac{\tau_2}{t_2}\right)^{k_2}\right)\frac{d\tau_2}{\tau_2}\right)\\
=\sum_{p=1}^{N_1-1}\Gamma\left(\frac{p}{k_1}\right)t_1^{p}k_2\int_{L_{d_2}}\frac{(\partial_{\tau_1}^{p}\omega)(0,\tau_2)}{p!}\exp\left( -\left(\frac{\tau_2}{t_2}\right)^{k_2}\right)\frac{d\tau_2}{\tau_2}=H_{1}(t_1,t_2)
\end{multline*}
and 
\begin{multline*}
k_1k_2\int_{L_{d_1}}\int_{L_{d_2}} P_{N_2,\tau_1}(\tau_2)\exp\left( -\left(\frac{\tau_1}{t_1}\right)^{k_1}-\left(\frac{\tau_2}{t_2}\right)^{k_2}\right)\frac{d\tau_2}{\tau_2}\\
=\sum_{q=1}^{N_2-1}\left(k_2\int_{L_{d_2}}\tau_2^q\exp\left( -\left(\frac{\tau_2}{t_2}\right)^{k_2}\right)\frac{d\tau_2}{\tau_2}\right)\left(k_1\int_{L_{d_1}}\frac{(\partial_{\tau_2}^q\omega)(\tau_1,0)}{q!}\exp\left( -\left(\frac{\tau_1}{t_1}\right)^{k_1}\right)\frac{d\tau_1}{\tau_1}\right)\\
=\sum_{q=1}^{N_2-1}\Gamma\left(\frac{q}{k_2}\right)t_2^{q}k_1\int_{L_{d_1}}\frac{(\partial_{\tau_2}^{q}\omega)(\tau_1,0)}{q!}\exp\left( -\left(\frac{\tau_1}{t_1}\right)^{k_1}\right)\frac{d\tau_1}{\tau_1}=H_{2}(t_1,t_2)
\end{multline*}
along with
\begin{multline*}
k_1k_2\int_{L_{d_1}}\int_{L_{d_2}}\left(\sum_{p=1}^{N_1-1}\sum_{q=1}^{N_2-1}\frac{\tau_1^p}{p!}\frac{\tau_2^q}{q!}(\partial_{\tau_1}^p\partial_{\tau_2}^q\omega)(0,0)\right)\exp\left(-\left(\frac{\tau_1}{t_1}\right)^{k_1}-\left(\frac{\tau_2}{t_2}\right)^{k_2}\right)\frac{d\tau_2}{\tau_2}\frac{d\tau_1}{\tau_1}\\
=\sum_{p=1}^{N_1-1}\sum_{q=1}^{N_2-1}\frac{\Gamma\left(\frac{p}{k_1}\right)\Gamma\left(\frac{q}{k_2}\right)}{p!q!}(\partial_{\tau_1}^p\partial_{\tau_2}^q\omega)(0,0)t_1^pt_2^q=H_3(t_1,t_2)
\end{multline*}

We arrive at the next integral representation
\begin{multline}
\mathcal{L}_{\boldsymbol{k}}^{\boldsymbol{d}}(\omega)(\boldsymbol{t})-H_1(t_1,t_2)-H_2(t_1,t_2)+H_3(t_1,t_2)\\
=k_1k_2\int_{L_{d_1}}\int_{L_{d_2}}\left[\omega(\tau_1,\tau_2)-P_{N_1,\tau_2}(\tau_1)-P_{N_2,\tau_1}(\tau_2)\right.\\
\left.+\sum_{p=1}^{N_1-1}\sum_{q=1}^{N_2-1}\frac{\tau_1^p}{p!}\frac{\tau_2^q}{q!}(\partial_{\tau_1}^p\partial_{\tau_2}^q\omega)(0,0)\right]
\exp\left(-\left(\frac{\tau_1}{t_1}\right)^{k_1}-\left(\frac{\tau_2}{t_2}\right)^{k_2}\right)\frac{d\tau_1}{\tau_1}\frac{d\tau_2}{\tau_2}\label{e1111}
\end{multline}

From the Taylor formula with integral remainder in one variable applied to the partial map $\tau_1\mapsto \omega(\tau_1,\tau_2)$ for every fixed $\tau_2\in S_{d_2}\cup D(0,\rho)$ we observe that
$$\omega(\tau_1,\tau_2)-P_{N_1,\tau_2}(\tau_1)=\int_0^{\tau_1}(\partial_{\tau_1}^{N_1}\omega)(t_1,\tau_2)\frac{(\tau_1-t_1)^{N_1-1}}{(N_1-1)!}dt_1,$$
for all $\tau_1\in S_{d_1}\cup D(0,\rho)$.

In the sequel, we provide an integral representation of the quantity
\begin{multline*}
\mathcal{A}_{N_1,N_2}(\tau_1,\tau_2)=\int_0^{\tau_1}(\partial_{\tau_1}^{N_1}\omega)(t_1,\tau_2)\frac{(\tau_1-t_1)^{N_1-1}}{(N_1-1)!}dt_1\\
-P_{N_2,\tau_1}(\tau_2)+\sum_{p=1}^{N_1-1}\sum_{q=1}^{N_2-1}\frac{\tau_1^p}{p!}\frac{\tau_2^q}{q!}(\partial_{\tau_1}^p\partial_{\tau_2}^q\omega)(0,0).
\end{multline*}
Namely, for each $1\le q\le N_2-1$, the map $\tau_1\mapsto (\partial_{\tau_2}^q\omega)(\tau_1,0)$ has the next Taylor expansion
\begin{equation}\label{e1124}(\partial_{\tau_2}^q\omega)(\tau_1,0)=\sum_{p=1}^{N_1-1}\frac{(\partial_{\tau_1}^p\partial_{\tau_2}^q\omega)(0,0)}{p!}\tau_1^p+
\int_0^{\tau_1}(\partial_{\tau_1}^{N_1}\partial_{\tau_2}^{q}\omega)(t_1,0)\frac{(\tau_1-t_1)^{N_1-1}}{(N_1-1)!}dt_1
\end{equation}
for all $\tau_1\in S_{d_1}\cup D(0,\rho)$.

We deduce from (\ref{e1124}) the next expansion 
\begin{multline*}
P_{N_2,\tau_1}(\tau_2)=\sum_{q=1}^{N_2-1}\left[\sum_{p=1}^{N_1-1}\frac{(\partial_{\tau_1}^p\partial_{\tau_2}^q\omega)(0,0)}{p!q!}\tau_1^p\right]\tau_2^q\\
+\sum_{q=1}^{N_2-1}\tau_2^q\int_0^{\tau_1}\frac{(\partial_{\tau_1}^{N_1}\partial_{\tau_2}^q\omega)(t_1,0)}{q!}\frac{(\tau_1-t_1)^{N_1-1}}{(N_1-1)!}dt_1.
\end{multline*}
As a result of the previous expression, we get
\begin{multline*}
\mathcal{A}_{N_1,N_2}(\tau_1,\tau_2)=\int_0^{\tau_1}(\partial_{\tau_1}^{N_1}\omega)(t_1,\tau_2)\frac{(\tau_1-t_1)^{N_1-1}}{(N_1-1)!}dt_1\\
-\sum_{q=1}^{N_2-1}\tau_2^q\int_0^{\tau_1}\frac{(\partial_{\tau_1}^{N_1}\partial_{\tau_2}^q\omega)(t_1,0)}{q!}\frac{(\tau_1-t_1)^{N_1-1}}{(N_1-1)!}dt_1.
\end{multline*}

Now, we notice the next Taylor expansion of the partial map $\tau_2\mapsto (\partial_{\tau_1}^{N_1}\omega)(t_1,\tau_2)$
$$(\partial_{\tau_1}^{N_1}\omega)(t_1,\tau_2)=\sum_{q=1}^{N_2-1}\frac{(\partial_{\tau_2}^{q}\partial_{\tau_1}^{N_1}\omega)(t_1,0)}{q!}\tau_2^q+\int_0^{\tau_2}\frac{(\tau_2-t_2)^{N_2-1}}{(N_2-1)!}(\partial_{\tau_2}^{N_2}\partial_{\tau_1}^{N_1}\omega)(t_1,t_2)dt_2.$$

From the previous expression, we reach the integral representation
\begin{multline*}
\mathcal{A}_{N_1,N_2}(\tau_1,\tau_2)=\int_0^{\tau_1}\left(\sum_{q=1}^{N_2-1}\frac{(\partial_{\tau_2}^{q}\partial_{\tau_1}^{N_1}\omega)(t_1,0)}{q!}\tau_2^q+\int_0^{\tau_2}\frac{(\tau_2-t_2)^{N_2-1}}{(N_2-1)!}(\partial_{\tau_2}^{N_2}\partial_{\tau_1}^{N_1}\omega)(t_1,t_2)dt_2\right)\\
\times\frac{(\tau_1-t_1)^{N_1-1}}{(N_1-1)!}dt_1-\sum_{q=1}^{N_2-1}\tau_2^q\int_0^{\tau_1} \frac{(\partial_{\tau_1}^{N_1}\partial_{\tau_2}^{q}\omega)(t_1,0)}{q!} \frac{(\tau_1-t_1)^{N_1-1}}{(N_1-1)!}dt_1\\
=\int_0^{\tau_1}\int_0^{\tau_2}(\partial_{\tau_2}^{N_2}\partial_{\tau_1}^{N_1}\omega)(h_1,h_2)\frac{(\tau_2-h_2)^{N_2-1}}{(N_2-1)!}\frac{(\tau_1-h_1)^{N_1-1}}{(N_1-1)!}dh_2dh_1
\end{multline*}
for all $\tau_1\in S_{d_1}\cup D(0,\rho)$, $\tau_2\in S_{d_2}\cup D(0,\rho)$.

We disclose upper bounds for the norm $\left\|\mathcal{A}_{N_1,N_2}(\tau_1,\tau_2)\right\|_{\mathbb{E}}$. On the ground of the estimate (\ref{e970}), we get
\begin{multline}\label{e1155}
\left\|\mathcal{A}_{N_1,N_2}(\tau_1,\tau_2)\right\|_{\mathbb{E}}\le \tilde{C}(M_1)^{N_1}(M_2)^{N_2}N_1!N_2!\\
\times\int_0^{|\tau_1|}\int_0^{|\tau_2|}\exp\left(\tilde{K}_1s_1^{k_1}+\tilde{K}_2s_2^{k_2}\right)\frac{(|\tau_1|-s_1)^{N_1-1}}{(N_1-1)!}\frac{(|\tau_2|-s_2)^{N_2-1}}{(N_2-1)!}ds_2ds_1,
\end{multline}
for all $\tau_j\in\tilde{S}_{d_j}\cup D(0,\tilde{\rho}_j)$, $j=1,2$. Besides, observe that
\begin{equation}\label{e1156}
\int_0^{|\tau_1|}\frac{(|\tau_1|-s_1)^{N_1-1}}{(N_1-1)!}ds_1=\frac{|\tau_1|^{N_1}}{N_1!},\quad \int_0^{|\tau_2|}\frac{(|\tau_2|-s_2)^{N_2-1}}{(N_2-1)!}ds_2=\frac{|\tau_2|^{N_2}}{N_2!}.
\end{equation}

From (\ref{e1155}) and (\ref{e1156}), it results that
\begin{equation}\label{e1160}
\left\|\mathcal{A}_{N_1,N_2}(\tau_1,\tau_2)\right\|_{\mathbb{E}}\le \tilde{C}(M_1)^{N_1}(M_2)^{N_2}\exp(\tilde{K}_1|\tau_1|^{k_1}+\tilde{K}_2|\tau_2|^{k_2})|\tau_1|^{N_1}|\tau_2|^{N_2},
\end{equation}
for all $\tau_j\in\tilde{S}_{d_j}\cup D(0,\tilde{\rho}_j)$, $j=1,2$. At last, combining the integral representation (\ref{e1111}) and (\ref{e1160}) we get the existence of positive constants $\tilde{\Delta}_1,\tilde{\Delta}_2$ such that
\begin{multline*}
\left\|\mathcal{L}_{\boldsymbol{k}}^{\boldsymbol{d}}(\omega)(\boldsymbol{t})-H_1(t_1,t_2)-H_2(t_1,t_2)+H_3(t_1,t_2)\right\|_{\mathbb{E}}\\
\le \tilde{C}k_1k_2\int_0^{+\infty}(M_1)^{N_1}s_1^{N_1}\exp(-\tilde{\Delta}_1\left(\frac{s_1}{|t_1|}\right)^{k_1})\frac{ds_1}{s_1}\int_0^{+\infty}(M_2)^{N_2}s_2^{N_2}\exp(-\tilde{\Delta}_2\left(\frac{s_2}{|t_2|}\right)^{k_2})\frac{ds_2}{s_2}\\
=\tilde{C}(M_1)^{N_1}(M_2)^{N_2}\Gamma\left(\frac{N_1}{k_1}\right)\Gamma\left(\frac{N_2}{k_2}\right)\frac{1}{\tilde{\Delta}_1^{\frac{N_1}{k_1}}}\frac{1}{\tilde{\Delta}_2^{\frac{N_2}{k_2}}}|t_1|^{N_1}|t_2|^{N_2},
\end{multline*}
for all $(t_1,t_2)\in T_1\times T_2$ provided that $S_j$ has small enough radius for $j=1,2$. The result follows.


\end{document}